\documentclass[reqno,11pt]{amsart}

\usepackage[T1]{fontenc}
\usepackage[utf8]{inputenc}
\usepackage{amsmath,amssymb,amsthm,mathtools}
\usepackage{mathrsfs}
\usepackage{enumitem}
\usepackage{microtype}
\usepackage{aliascnt}
\usepackage[hidelinks]{hyperref}
\usepackage{cleveref}

\newtheorem{theorem}{Theorem}[section]
\newaliascnt{prop}{theorem}
\newtheorem{prop}[prop]{Proposition}
\aliascntresetthe{prop}
\newaliascnt{lem}{theorem}
\newtheorem{lem}[lem]{Lemma}
\aliascntresetthe{lem}
\newaliascnt{cor}{theorem}
\newtheorem{cor}[cor]{Corollary}
\aliascntresetthe{cor}
\theoremstyle{definition}
\newaliascnt{defn}{theorem}
\newtheorem{defn}[defn]{Definition}
\aliascntresetthe{defn}
\newaliascnt{example}{theorem}

\aliascntresetthe{example}
\theoremstyle{remark}
\newaliascnt{remark}{theorem}
\newtheorem{remark}[remark]{Remark}
\aliascntresetthe{remark}
\crefname{prop}{Proposition}{Propositions}
\crefname{lem}{Lemma}{Lemmas}
\crefname{cor}{Corollary}{Corollaries}
\crefname{defn}{Definition}{Definitions}
\crefname{remark}{Remark}{Remarks}

\numberwithin{equation}{section}
\newcommand{\M}{\operatorname{M}}
\newcommand{\SL}{\operatorname{SL}}
\newcommand{\GL}{\operatorname{GL}}
\newcommand{\SO}{\operatorname{SO}}
\newcommand{\vol}{\operatorname{vol}}
\newcommand{\covol}{\operatorname{covol}}

\newcommand{\Leb}{\operatorname{Leb}}

\newcommand{\supp}{\operatorname{supp}}

\newcommand{\cM}{\mathcal M}

\title[Quantitative Oppenheim in signature $(2,2)$]{Quantitative Oppenheim in signature $(2,2)$ \\ via determinant values}
\author{Wooyeon Kim}
\address{Korea Institute for Advanced Study, Seoul, Republic of Korea}
\email{wooyeonkim@kias.re.kr}
\author{Hee Oh}
\address{Department of Mathematics, Yale University, New Haven, CT 06511}
\email{hee.oh@yale.edu}

\begin{document}
\begin{abstract}
We give a new proof of the quantitative Oppenheim theorem in signature
$(2,2)$, originally proved by Eskin--Margulis--Mozes, by recasting the
problem as one about determinant values on lattices in $\M_2(\mathbb R)$.
\end{abstract}

\maketitle
\section{Introduction}
The Oppenheim conjecture, proved by Margulis
(\cite{Margulis:1987}, \cite{Margulis:1989}), asserts that if $Q$ is an indefinite
quadratic form in $n\ge3$ variables which is not proportional to a rational
form, then $Q(\mathbb Z^n)$ is dense in $\mathbb R$. A quantitative
refinement asks, for fixed real numbers $a<b$, for the asymptotic behavior
as $T\to\infty$ of
$$
N_Q(a,b;T)
:=
\#\{v\in\mathbb Z^n:\|v\|<T,\ a<Q(v)<b\}.
$$
This quantitative Oppenheim problem was established in higher signatures by
Dani--Margulis \cite{dani-margulis:1993} and
Eskin--Margulis--Mozes \cite{eskin-margulis-mozes:1998}. The low signatures
$(2,1)$ and $(2,2)$ require additional Diophantine hypotheses: the
signature $(2,2)$ case was proved by Eskin--Margulis--Mozes
\cite{eskin-margulis-mozes:2005}, and the signature $(2,1)$ case was
proved by the first-named author in \cite{Kim}.

The purpose of this paper is to give a new proof of the $(2,2)$ case by
recasting it as a problem about determinant values on lattices in
$\M_2(\mathbb R)$. This determinant formulation is useful for two reasons.
First, it is the simplest member of the higher-dimensional determinant-value
problem studied in our companion paper \cite{KimOh_det}; in dimension $2$, the
main dynamical mechanism remains visible while many of the higher-dimensional
complications disappear. Second, it provides a transparent setting for the
modified-height and avoidance method introduced in \cite{Kim}. Thus the
present paper gives a streamlined model case for the method, while also
recovering the $(2,2)$ theorem of Eskin--Margulis--Mozes.

Let
$$
Q_0\begin{pmatrix}x&y\\ z&w\end{pmatrix}
=xw-yz.
$$
Then $Q_0$ is a quadratic form of signature $(2,2)$. Conversely, every
real quadratic form of signature $(2,2)$ is of the form
$Q_0\circ L$ for some real linear isomorphism
$L:\mathbb R^4\to\M_2(\mathbb R)$. Thus the fixed-form, varying-lattice
problem for the determinant on $\M_2(\mathbb R)$ is equivalent to the
usual formulation of the signature $(2,2)$ quantitative Oppenheim problem.

Let $\Lambda<\M_2(\mathbb R)$ be a lattice, and fix a norm
$\|\cdot\|$ on $\M_2(\mathbb R)$ invariant under left and right
multiplication by $\SO(2)$. For $a<b$ and $T>0$, set
$$
N_\Lambda(a,b;T)
:=
\#\{v\in\Lambda:\|v\|<T,\ a<\det v<b\}\qquad\text{and}
$$
$$
N_\Lambda^\times(a,b;T)
:=
\#\{v\in\Lambda:\|v\|<T,\ a<\det v<b,\ \det v\ne0\}.
$$
The distinction matters when $0\in(a,b)$, because lattice points with zero
determinant may contribute on the same $T^2$-scale as the nonsingular
points.

\subsection*{Main theorem}

We call $\Lambda$ \emph{determinant-rational} if
$$
\det(\Lambda)\subset \lambda\mathbb Q
\qquad\text{for some }\lambda\in\mathbb R^\times.
$$ 
Equivalently, $\Lambda$ is determinant-rational precisely when, for any $\mathbb Z$-basis $\mathcal B=(v_1,v_2, v_3,v_4)$, the determinant form
\begin{equation}\label{eq:FLambdadef}
Q_{\Lambda,\mathcal B}(x_1,x_2, x_3,x_4)=\det(x_1v_1+x_2v_2+x_3v_3+x_4v_4)
\end{equation}
is proportional to a polynomial with rational coefficients. This is the arithmetic obstruction to density of determinant values.

We impose a Diophantine condition excluding lattices which are too well
approximated by split rational determinant forms. The precise definition is
given in \Cref{def:diophantine}; under the quadratic-form identification, it
is the non-EWAS condition of Eskin--Margulis--Mozes \cite{eskin-margulis-mozes:2005}.

\begin{theorem}\label{thm:main}
Let $\Lambda<\M_2(\mathbb R)$ be a Diophantine lattice which is not
determinant-rational. There exist
$C_{\|\cdot\|}>0$ and $c_{\Lambda}^{\operatorname{sing}}\ge 0$ depending only on the norm, such that, for every
$a<b$,
\begin{equation}\label{eq:main-nonsingular}
N_\Lambda^\times(a,b;T)
\sim
\frac{C_{\|\cdot\|}}{\covol(\Lambda)}(b-a)T^2;
\end{equation}
\begin{equation}\label{eq:main-full}
N_\Lambda(a,b;T)
\sim
\left(
\frac{C_{\|\cdot\|}}{\covol(\Lambda)}(b-a)
+
c_\Lambda^{\operatorname{sing}}\mathbf 1_{\{0\in(a,b)\}}
\right)T^2.
\end{equation}
Moreover, $c_\Lambda^{\operatorname{sing}}>0$ if and only if $\Lambda$ contains
a rank-two $\mathbb Z$-submodule on which the determinant vanishes identically.
\end{theorem}

Here covolumes are computed with respect to the Lebesgue measure induced by
the trace inner product $\langle x,y\rangle=\operatorname{tr}(xy^{\intercal})$.
The constants $C_{\|\cdot\|}$ and $c_\Lambda^{\operatorname{sing}}$  are  characterized by the following:
\begin{equation}\label{eq:volume-asymptotic-intro}
\vol\{X\in\M_2(\mathbb R):\|X\|<T,\ a<\det X<b\}
\sim
C_{\|\cdot\|}(b-a)T^2;
\end{equation}
\begin{equation}\label{eq:sing-explicit}
c_\Lambda^{\operatorname{sing}}
=
\sum_{V\in\mathcal I_\Lambda}
\frac{\vol_V\{v\in V:\|v\|<1\}}
{\covol_V(\Lambda\cap V)},
\end{equation}
where $\mathcal I_\Lambda$ denotes the set of two-dimensional subspaces $V\subset\{\det=0\}$ for which $\Lambda\cap V$ is a lattice in $V$. Here $\vol_V$ is the Lebesgue measure on $V$ induced by the trace inner product, and $\covol_V$ is the corresponding covolume.
The finiteness of $\mathcal I_\Lambda$ for determinant non-rational lattices
is a special feature of the $2\times2$ case (see \Cref{prop:finite-isotropic-planes-d2}).

The following is a special case of Theorem \ref{thm:main} (\cite[\S13]{KimOh_det}).
\begin{cor}\label{cor:algebraic}
Let $\Lambda<\M_2(\mathbb R)$ be a lattice  whose elements have
algebraic entries. If $\Lambda$ is not determinant-rational, then the
conclusions of \Cref{thm:main} hold.
\end{cor}

\subsection*{Quadratic-form formulation}

Let $Q$ be a quadratic form of signature $(2,2)$ on $\mathbb R^4$.
Choose a real linear isomorphism $L:\mathbb R^4\to\M_2(\mathbb R)$ such that
$$
Q(x)=\det(Lx).
$$
Applying \Cref{thm:main} to the lattice $L(\mathbb Z^4)$  gives the following equivalent formulation. The Diophantine condition on $Q$ given in \eqref{dio} is the same as the non-EWAS condition of \cite{eskin-margulis-mozes:2005}.

\begin{theorem}\label{thm:quadratic-form-version}
Let $Q$ be a quadratic form of signature $(2,2)$ which is not proportional to a rational form. Suppose that $Q$ satisfies
the Diophantine condition.  Then there exist
$c_{Q}>0$ and
$c_{Q}^{\operatorname{sing}}\ge0$, depending on $\|\cdot\|$, such that  for every $a<b$,
$$
\#\{v\in\mathbb Z^4:\|v\|<T,\ a<Q(v)<b\}
\sim
\left(
c_Q(b-a)
+
c_Q^{\operatorname{sing}}\mathbf 1_{\{0\in(a,b)\}}
\right)T^2.
$$
Moreover, $c_Q^{\operatorname{sing}}>0$ precisely when $Q$ vanishes
identically on a rational two-dimensional subspace of $\mathbb R^4$.
\end{theorem}

The principal difference between our proof and that of
Eskin--Margulis--Mozes lies in the derivation of the uniform moment bound
needed for the counting argument. Their modified height, obtained by
omitting quasi-null planes, is well suited to counting, but it does not
directly satisfy the Margulis-type inequality needed for our iteration.
We therefore introduce a second, auxiliary height. Following the strategy
of \cite{Kim}, we establish the required inequality for this auxiliary
height away from an exceptional set and control the exceptional set using
a quantitative avoidance estimate.

\subsection*{Organization of the paper}
Section~2 introduces the notation used throughout the paper and outlines the
proof of the main theorem. Section~3 discusses determinant-rationality and
proves the finiteness of rational isotropic planes. Section~4 formulates the
Diophantine condition and introduces the modified height. Sections~5--8
establish the uniform moment bound: Section~5 proves the local contraction
estimates, Section~6 globalizes these estimates to the auxiliary height,
Section~7 establishes avoidance of the exceptional set, and Section~8
combines contraction and avoidance through an iterative argument. Section~9
shows that the contribution of non-isotropic quasi-null planes vanishes in
the limit and proves the modified Siegel-transform limit. Finally,
Section~10 converts this dynamical limit into the determinant-counting
asymptotic and incorporates the singular contribution.

\subsection*{Acknowledgements}
W.K. is supported
by KIAS individual grant HP101301. H.O. is partially supported by NSF grant DMS-2450703.

\subsection*{AI disclosure}
The authors used AI tools to assist with checking proofs  and editing the manuscript. All mathematical ideas and contributions are due to the authors.
\section{Notation and proof ideas}

\subsection*{Notation}

Throughout the paper, we set
\begin{equation}\label{gg}
G=\SL_4(\mathbb R),\qquad
\Gamma=\SL_4(\mathbb Z),\qquad
X=G/\Gamma.
\end{equation}
We denote by $m_X$ the $G$-invariant probability measure on $X$, and
write $[g]=g\Gamma\in X$ for $g\in G$.

We identify $\M_2(\mathbb R)$ with $\mathbb R^4$ by identifying the
elementary matrices
$
E_{11},\ E_{12},\ E_{21},\ E_{22}
$
with the standard basis vectors $e_1,e_2,e_3,e_4$, respectively. Under
this identification, $X$ is the space of unimodular lattices in
$\M_2(\mathbb R)$. We equip $\M_2(\mathbb R)$ with the trace inner
product
$$
\langle v,w\rangle=\operatorname{tr}(vw^{\intercal}),\qquad(v, w\in \M_2(\mathbb R))
$$
and use the induced Euclidean norms on $\M_2(\mathbb R)$ and its exterior
powers in all height and representation-theoretic constructions. 

The group $\SL_2(\mathbb R)\times\SL_2(\mathbb R)$ acts on
$\M_2(\mathbb R)$ by
\begin{equation}\label{gh}
(g_1,g_2)\cdot v=g_1vg_2^{\intercal}
\qquad
(g_1,g_2\in\SL_2(\mathbb R),\ v\in\M_2(\mathbb R)).
\end{equation}
Under the identification $\M_2(\mathbb R)\simeq\mathbb R^4$, this action
is given by the tensor-product representation $g_1\otimes g_2$. We denote
its image in $G$ by
$$
H
=
\SL_2(\mathbb R)\otimes\SL_2(\mathbb R)
=
\{g_1\otimes g_2:g_1,g_2\in\SL_2(\mathbb R)\}
<G.
$$
Thus the $\SL_2(\mathbb R)\times\SL_2(\mathbb R)$-orbit of
$v\in\M_2(\mathbb R)$ corresponds to the $H$-orbit of $v\in\mathbb R^4$.
We write $e$ for the identity element of $G$.

Write
$$
\M_2(\mathbb R)=\mathsf C\otimes\mathsf R,
$$
where $\mathsf C$ and $\mathsf R$ are the column and row copies of
$\mathbb R^2$, respectively. The action of $H$ extends naturally to
$\wedge^i\M_2(\mathbb R)$ for $1\le i\le3$. The $H$-modules
$$
\wedge^1\M_2(\mathbb R)=\mathsf C\otimes\mathsf R
\qquad\text{and}\qquad
\wedge^3\M_2(\mathbb R)\simeq\mathsf C^*\otimes\mathsf R^*
$$
are irreducible, whereas
\begin{equation}\label{decom}
\wedge^2\M_2(\mathbb R)
=
\mathcal M_c\oplus\mathcal M_r
\end{equation}
is the decomposition into irreducible $H$-modules, where
\begin{equation}\label{decom2}
\mathcal M_c
=
\operatorname{Sym}^2(\mathsf C)\otimes\wedge^2\mathsf R,
\qquad
\mathcal M_r
=
\wedge^2\mathsf C\otimes\operatorname{Sym}^2(\mathsf R).
\end{equation}
See, for instance, \cite[Section~4.2]{FultonHarris} or
\cite[Chapter~I]{FultonYoung}. The right $\SL_2(\mathbb R)$-factor acts
trivially on $\mathcal M_c$, while the left $\SL_2(\mathbb R)$-factor
acts trivially on $\mathcal M_r$. Let $\pi_c$ and $\pi_r$ denote the
orthogonal projections onto $\mathcal M_c$ and $\mathcal M_r$,
respectively.

Let $V<\M_2(\mathbb R)$ be an $i$-dimensional subspace for $1\le i\le 3$. A nonzero vector
$$
\mathsf w_V=v_1\wedge\cdots\wedge v_i
\in\wedge^i\M_2(\mathbb R)
$$
is called a \emph{Pl\"ucker vector} of $V$ if
$\{v_1,\ldots,v_i\}$ is a basis of $V$. It is unique up to multiplication
by a nonzero scalar. If $V$ is a two-plane, then
$$
\det|_V\equiv0
\quad\Longleftrightarrow\quad
\mathsf w_V\in\mathcal M_c\cup\mathcal M_r.
$$

Let $\Delta<\M_2(\mathbb R)$ be a lattice. An $i$-dimensional subspace
$V$ is called $\Delta$-rational if $V\cap\Delta$ is a lattice in $V$.
For such a subspace, a vector
$$
\mathsf w_{\Delta,V}
=
v_1\wedge\cdots\wedge v_i
\in\wedge^i\M_2(\mathbb R)
$$
is called a \emph{primitive Pl\"ucker vector} if
$\{v_1,\ldots,v_i\}$ is a $\mathbb Z$-basis of $V\cap\Delta$. It is
unique up to sign, and its norm is the covolume of $V\cap\Delta$ in $V$.

For $1\le i\le3$, define
$$
\alpha_i(\Delta)
=
\sup_{\substack{V\text{ $\Delta$-rational},\; \dim V=i}}
\|\mathsf w_{\Delta,V}\|^{-1},
$$
and set
$$
\alpha(\Delta)=\max_{1\le i\le3}\alpha_i(\Delta).
$$

Finally, identifying $\SL_2(\mathbb R)\times\SL_2(\mathbb R)$ with its
image in $H$, we set
$$
K=\SO(2)\times\SO(2),\qquad
b_t=
\begin{pmatrix}
e^{-t}&0\\
0&e^t
\end{pmatrix},
\qquad
a_t=(b_t,b_t).
$$
We equip $K$ with probability Haar measure $dk$, obtained by pushing
forward the product probability Haar measure on
$\SO(2)\times\SO(2)$.
The two parameters $t$ and $T$ are related by $ T=e^{2t}$.

We write
$$
u_\xi=
\begin{pmatrix}
1&0\\
\xi&1
\end{pmatrix},
\qquad
n_{\xi_1,\xi_2}=(u_{\xi_1},u_{\xi_2}),
$$
and set
$$
N=\{n_{\xi_1,\xi_2}:\xi_1,\xi_2\in\mathbb R\},
\qquad
B_N(R)=\{n_{\xi_1,\xi_2}:|\xi_1|,|\xi_2|\le R\}.
$$
Under the identification $N\simeq\mathbb R^2$, the set $B_N(R)$ is the
ball of radius $R$ for the sup norm. We normalize $dn$ so that $B_N(1)$
has measure one.

\subsection*{Proof ideas}\label{idea}
For this discussion, after rescaling, we assume that $\Lambda\in X$. The
general case is recovered from the normalization in
Section~\ref{sec:det-rational}.
For $f\in C_c(\M_2(\mathbb R)\smallsetminus\{0\})$, define its Siegel transform by
$$
\widehat f(\Delta)
=
\sum_{0\ne v\in\Delta}f(v),
\qquad \Delta\in X.
$$
After smoothing the conditions $\|v\|<T$ and $a<\det v<b$, we reduce the
counting problem to averages of the form
$$
\int_K\widehat f(a_tk\Lambda)\,dk.
$$

By Ratner's orbit-closure theorem \cite{ratner:1991}, the assumption that
$\Lambda$ is not determinant-rational implies that $H\Lambda$ is dense in
$X$. Shah's theorem on expanding translates \cite{shah} then gives, for
every  continuous compactly supported function $F$ on $X$,
$$
\lim_{t\to\infty}\int_KF(a_tk\Lambda)\,dk
=
\int_XF\,dm_X.
$$
The difficulty is that $\widehat f$ is unbounded in the cusp. The usual
Lipschitz principle gives
$|\widehat f(\Delta)|\ll_f\alpha(\Delta)$, so it would suffice to prove
$$
\sup_{t\ge0}
\int_K\alpha(a_tk\Lambda)^{1+\delta}\,dk
<\infty
\qquad\text{for some $\delta>0$}.
$$
This estimate, however, fails in the critical exterior degree $2$.
We encounter two related obstructions. First, a $\Lambda$-rational
isotropic two-plane contains a rank-two sublattice and may therefore
contribute on the same $T^2$-scale as the nonsingular points. We remove
such planes from the nonsingular count and treat them separately. A
special feature of the $2\times2$ setting is that only finitely many such
planes arise: a lattice that is not determinant-rational has at most two
rational column-isotropic planes and at most two rational row-isotropic
planes.
The second obstruction comes from rational planes that are close to being
isotropic. Fix $\eta>0$ and $M>1$. We call a $\Delta$-rational two-plane
$(\eta,M)$-quasi-null if its primitive Pl\"ucker vector $w$ satisfies
$$
\min\{\|w-\pi_cw\|,\|w-\pi_rw\|\}
\le
\eta\|w\|^{-M}.
$$
Near either exceptional summand $\mathcal M_c$ or $\mathcal M_r$, one of
the two $\SL_2(\mathbb R)$-factors acts almost trivially on the dominant
component, and the degree-$2$ contraction estimate degenerates.

The Diophantine hypothesis gives a uniform finite-scale bound for such
planes. More precisely, we can choose $\eta$ and $M$ so that, for every
$R\ge1$,
$$
\#\left\{
V:
\|\mathsf w_{\Lambda,V}\|\le R,\min\{\|\mathsf w_{\Lambda,V}-\pi_c\mathsf w_{\Lambda,V}\|,
      \|\mathsf w_{\Lambda,V}-\pi_r\mathsf w_{\Lambda,V}\|\}
\le\eta R^{-M}\right\}
\le4.
$$
Indeed, five such planes would allow us to reconstruct a rational split
quadratic form approximating the determinant form more closely than the
Diophantine condition permits.

We therefore replace the ordinary height by the modified height
$\widehat\alpha_{\eta,M}$ introduced in
\Cref{def:modified-height-d2}.  Its degree-$1$ and degree-$3$ terms agree
with those of $\alpha$, while in degree $2$ we omit the quasi-null
planes.

To prove a moment estimate for $\widehat\alpha_{\eta,M}$, we introduce the
auxiliary height $\widetilde\alpha_{\eta,M,\theta}$ of
\Cref{def:auxiliary-height-d2}, where $\theta>0$ is small. In the critical
degree, this height uses the weighted function $\phi_\theta$ defined in
\eqref{eq:d2-local-weight}. The weight compensates for the degeneration of
the degree-$2$ contraction estimate near $\mathcal M_c$ and
$\mathcal M_r$, while quasi-null planes are retained with a subcritical
exponent. The comparison in \eqref{eq:auxiliary-comparisons} gives that for every $h\in H$
and $\Lambda\in X$,
$$
\widehat\alpha_{\eta,M}(h;\Lambda)^{1+\theta}
\ll_\theta
\widetilde\alpha_{\eta,M,\theta}(h;\Lambda).
$$

We isolate an exceptional set consisting of points for which a
non-quasi-null rational two-plane of controlled covolume is carried
extremely close to $\mathcal M_c$ or $\mathcal M_r$. Outside this set, we
prove the Margulis inequality
$$
\int_{B_N(1)}
\widetilde\alpha_{\eta,M,\theta}(a_snh;\Lambda)\,dn
\le
Ce^{-cs}\widetilde\alpha_{\eta,M,\theta}(h;\Lambda)
+
Ce^{Cs}
$$
for some constants $c,C>0$ (Proposition \ref{prop:globalcontraction}).

We control the exceptional set using quantitative nondivergence and the
Diophantine hypothesis. The obstruction alternative in the quantitative
nondivergence estimate in \cite{kleinbock:2008} would produce a rational subspace of
$\wedge^2\M_2(\mathbb R)$ lying very close to $\mathcal M_c$ or
$\mathcal M_r$. From such a subspace, one reconstructs a rational split
quadratic form approximating the determinant form too well, contradicting
the Diophantine condition. 
We then combine
contraction outside the exceptional set with quantitative avoidance in a
stopping-time iteration. This yields some $\delta>0$ such that
\begin{equation}\label{di}
\sup_{t\ge0}
\int_K
\widehat\alpha_{\eta,M}(a_tk;\Lambda)^{1+\delta}\,dk
<\infty.
\end{equation}
This derivation of \eqref{di} is the principal difference between our
proof and that of \cite{eskin-margulis-mozes:2005}. Their argument works
directly with the $K$-averages: they approximate the bad sets associated
with short rational planes by rectangles, divide them into
\emph{old} and \emph{new} families, and control them through shrinking,
overlap, and nesting arguments.

We must still account for the planes omitted from
$\widehat\alpha_{\eta,M}$. The finite-scale bound leaves at most four
quasi-null planes in each dyadic Pl\"ucker-height range, while a critical
first-moment estimate shows that the contribution of each fixed
non-isotropic quasi-null plane tends to zero. Summing over the dyadic
ranges, we conclude that their total contribution is negligible.

The modified Lipschitz principle and the uniform moment bound therefore give uniform integrability for the Siegel transform after removing the vectors contained in the exact rational isotropic planes. Denote this modified Siegel transform by $\widehat f_{\operatorname{ni}}$. We may then truncate in the cusp, apply Shah's theorem to the bounded truncation, and use Siegel's mean-value formula \cite{siegel1945mean} to obtain
$$
\lim_{t\to\infty}
\int_K\widehat f_{\operatorname{ni}}(a_tk;\Lambda)\,dk
=
\frac{1}{\covol(\Lambda)}
\int_{\M_2(\mathbb R)}f(x)\,dx.
$$

Finally, a fiber-kernel argument converts this limit into the
determinant-counting asymptotic for the non-isotropic points, and a
dyadic-shell estimate removes the remaining singular-value truncation.
We count the finitely many rational isotropic planes separately using the
standard two-dimensional lattice-point asymptotic, which yields the
singular term in
\eqref{eq:main-full}.
\begin{remark}
Lindenstrauss, Mohammadi, and Wang \cite{LMW-flat-torus} establish
pair-correlation asymptotics with polynomial error terms for the spectra
of flat two-dimensional tori. They realize their arithmetic family on the
product space
$(\SL_2(\mathbb R)/\SL_2(\mathbb Z))^2$, which allows them to treat the
cusp more directly. By contrast, we treat general Diophantine quadratic
forms of signature $(2,2)$ in the $\SL_4$-lattice space.
\end{remark}

\section{Determinant rationality and finiteness of rational isotropic planes}
\label{sec:det-rational}
We recall from \cite{KimOh_det} the determinant-rationality,
orbit-closure, and rationality-obstruction results needed below, and give
the short dimension-$2$ argument proving the finiteness of rational
isotropic planes. By rescaling, it suffices to prove the main theorem for
unimodular lattices; hence, throughout the proof, we assume that
$\Lambda\in X$.

By \cite[Theorem~2.7]{KimOh_det}, a lattice
$\Lambda<\M_2(\mathbb R)$ is determinant-rational if and only if, for
some\/ (equivalently, every) $\mathbb Z$-basis $\mathcal B$ of $\Lambda$,
the determinant form $Q_{\Lambda,\mathcal B}$ is proportional to a
polynomial with rational coefficients.

The following orbit-closure dichotomy follows from Ratner's theorem
\cite{ratner:1991}; see \cite[Theorem~2.10]{KimOh_det}.

\begin{theorem}\label{thm:orbit-dichotomy-d2}
Let $\Lambda\in X$. Then $H\Lambda$ is closed in $X$ if and only if
$\Lambda$ is determinant-rational; otherwise, $H\Lambda$ is dense in
$X$.
\end{theorem}

We next consider rational two-planes contained in the singular set. We
call a two-dimensional subspace $V<\M_2(\mathbb R)$
\emph{column-isotropic} if all columns of all matrices in $V$ lie in a
common line in $\mathbb R^2$, and \emph{row-isotropic} if the analogous
condition holds for rows. The decomposition \eqref{decom} gives the
following characterization; see \cite[Lemma~3.2]{KimOh_det}.

\begin{lem}\label{lem:isotropic-plucker-d2}
Let $V<\M_2(\mathbb R)$ be a two-dimensional subspace, and let
$\mathsf w_V\in\wedge^2\M_2(\mathbb R)$ be a Pl\"ucker vector of $V$.
Then $V\subset\{\det=0\}$ if and only if $V$ is column-isotropic or
row-isotropic. Moreover,
$$
\begin{aligned}
V\text{ is column-isotropic}
&\quad\Longleftrightarrow\quad
\mathsf w_V\in\mathcal M_c,\\
V\text{ is row-isotropic}
&\quad\Longleftrightarrow\quad
\mathsf w_V\in\mathcal M_r.
\end{aligned}
$$
\end{lem}

To bound the number of rational isotropic planes, we use the following
rationality obstruction from \cite[Proposition~3.3]{KimOh_det}.

\begin{lem}[Rationality obstruction]
\label{lem:rational-ruling-d2}
Let $\Delta<\M_2(\mathbb R)$ be a lattice. If
$\mathcal M_c\cap\wedge^2\Delta$ is a lattice in $\mathcal M_c$, then
$\Delta$ is determinant-rational. The same conclusion holds with
$\mathcal M_r$ in place of $\mathcal M_c$.
\end{lem}

This obstruction yields the required finiteness statement.

\begin{prop}[Finiteness of rational isotropic planes]
\label{prop:finite-isotropic-planes-d2}
Let $\Lambda<\M_2(\mathbb R)$ be a lattice that is not
determinant-rational. Then $\Lambda$ has at most two
$\Lambda$-rational column-isotropic planes and at most two
$\Lambda$-rational row-isotropic planes. Consequently,
$$
\#\mathcal I_\Lambda\le4.
$$
\end{prop}

\begin{proof}
We prove the assertion for column-isotropic planes; the row-isotropic case
is identical. Put $e_{ij}=e_i\wedge e_j$. If the common column line is
spanned by $(s,t)^{\intercal}$, then the corresponding column-isotropic
plane is spanned by
$
se_1+te_3$ and $se_2+te_4,
$
and hence has a Pl\"ucker vector proportional to
$$
s^2e_{12}+st(e_{14}-e_{23})+t^2e_{34}.
$$
Any three distinct lines in $\mathbb R^2$ give
linearly independent vectors in $\mathcal M_c$. Indeed, after writing
$x=t/s$ whenever $s\ne0$, this is the usual Vandermonde determinant for
the vectors $(1,x,x^2)$; the case $s=0$ is immediate.
Suppose that $\Lambda$ has three distinct $\Lambda$-rational
column-isotropic planes. Their primitive Pl\"ucker vectors lie in
$\wedge^2\Lambda$ and, by the preceding calculation, form a basis of
$\mathcal M_c$. It follows that
$\mathcal M_c\cap\wedge^2\Lambda$ is a lattice in $\mathcal M_c$.
\Cref{lem:rational-ruling-d2} then implies that $\Lambda$ is
determinant-rational, contrary to the hypothesis. Thus $\Lambda$ has at
most two $\Lambda$-rational column-isotropic planes. The row-isotropic
case is identical.
\end{proof}

\begin{remark}
In the language of quadratic forms, the proposition says that a form of
signature $(2,2)$ that is not proportional to a rational form has at most
four rational two-dimensional totally isotropic subspaces; compare
\cite[Lemma~10.3]{eskin-margulis-mozes:2005}.
\end{remark}

\section{The Diophantine condition and the modified height}
\label{sec:diophantine-height}
The contraction in exterior degree $2$ becomes weaker when the
Pl\"ucker vector of a rational two-plane lies very close to one of the
exceptional summands $\mathcal M_c$ or $\mathcal M_r$. To control this, we
impose a Diophantine condition on approximation by rational split
quadratic forms. We then derive the finite-scale bound needed later and
define a modified height by omitting the corresponding quasi-null planes
from the degree-$2$ term.

Recall the determinant form $Q_{\Lambda,\mathcal B}$ defined in
\eqref{eq:FLambdadef}. Let $\mathscr S_{\mathbb Q}$ denote the set of
projective classes of nonzero rational quadratic forms that are
$\mathbb Q$-equivalent to $x_1x_4-x_2x_3$. We refer to these classes as
\emph{rational split forms}.

Fix the Euclidean norm on the ten-dimensional coefficient space of
quadratic forms on $\mathbb R^4$. For nonzero real quadratic forms $Q$ and
$Q'$, define
$$
\operatorname{dist}([Q],[Q'])
=
\min_{\varepsilon\in\{\pm1\}}
\left\|
\frac{Q}{\|Q\|}
-
\varepsilon\frac{Q'}{\|Q'\|}
\right\|.
$$
If $Q_0$ has rational coefficients, let $\operatorname{ht}(Q_0)$ be the
Euclidean norm of the primitive integral coefficient vector spanning the
line $\mathbb Q Q_0$. 

\begin{defn}[Split-Diophantine lattices]\label{def:diophantine}
Let $\eta>0$ and $M>1$. A lattice $\Lambda<\M_2(\mathbb R)$ is
\emph{$(\eta,M)$-split-Diophantine} if, for some $\mathbb Z$-basis
$\mathcal B$ of $\Lambda$,
$$
\operatorname{dist}([Q_{\Lambda,\mathcal B}],[Q_0])
\ge
\eta\,\operatorname{ht}(Q_0)^{-M}
\qquad
([Q_0]\in\mathscr S_{\mathbb Q}).
$$
We call $\Lambda$ \emph{Diophantine} if it is
$(\eta,M)$-split-Diophantine for some $\eta>0$ and $M>1$.
\end{defn}
This property is independent of the choice of $\mathcal B$, after
adjusting $\eta$. Accordingly, a nondegenerate quadratic form $Q$ of
signature $(2,2)$ on $\mathbb R^4$ is called
\emph{$(\eta,M)$-Diophantine} if
\begin{equation}\label{dio}
\operatorname{dist}([Q],[Q_0])
\ge
\eta\,\operatorname{ht}(Q_0)^{-M}
\qquad
([Q_0]\in\mathscr S_{\mathbb Q}),
\end{equation}
and \emph{Diophantine} if this holds for some $\eta>0$ and $M>1$.
If $Q(x)=\det(gx)$ for some
$g:\mathbb R^4\to\M_2(\mathbb R)$, then $Q$ is
Diophantine if and only if $g\mathbb Z^4$ is Diophantine. This is
the split-form version of the non-EWAS condition in
\cite[Definition~1.2]{eskin-margulis-mozes:2005}. 

We next translate this condition into a finite-scale statement about
rational two-planes. For a lattice $\Lambda<\M_2(\mathbb R)$, $R\ge1$, and
$0<\varepsilon<1$, let $\mathcal Q_\Lambda(R,\varepsilon)$ be the
collection of $\Lambda$-rational two-planes $V$ whose primitive Pl\"ucker
vector $w=\mathsf w_{\Lambda,V}$ satisfies
$$
\|w\|\le R,
\qquad
\min\{\|w-\pi_cw\|,\|w-\pi_rw\|\}\le\varepsilon.
$$
Thus $\mathcal Q_\Lambda(R,\varepsilon)$ consists of rational two-planes of
Pl\"ucker height at most $R$ that lie within $\varepsilon$ of one of the
two exceptional summands.

Among any five such planes, at least three lie close to the same
exceptional summand. If the approximation is sufficiently strong relative
to their height, these three planes are pairwise transverse, and we can
reconstruct from them a rational split quadratic form up to scale. We
only need polynomial dependence on the height; the exponent $18$ below
has no particular significance.

\begin{lem}[Reconstruction from five nearly isotropic planes]
\label{lem:d2-reconstruction-from-quasinull}
Let $\Lambda<\M_2(\mathbb R)$ be a lattice and let $\mathcal B$ be a
$\mathbb Z$-basis of $\Lambda$. There exists $C=C(\mathcal B)>0$ such
that, for every $R\ge1$ and $0<\varepsilon<1$, if
$\#\mathcal Q_\Lambda(R,\varepsilon)\ge5$, then there exists a rational
split quadratic form $P$ satisfying
\begin{equation}\label{eq:rational-split-reconstruction}
\operatorname{ht}(P)\le CR^{18},
\qquad
\operatorname{dist}([Q_{\Lambda,\mathcal B}],[P])
\le CR^{18}\varepsilon.
\end{equation}
\end{lem}

\begin{proof}
 After increasing $C$, we
may assume that
$
\varepsilon\le cR^{-18}
$
for a sufficiently small constant $c>0$; otherwise the conclusion follows
by taking any fixed rational split form.

Choose five distinct planes in $\mathcal Q_\Lambda(R,\varepsilon)$. By the
pigeonhole principle, three of them, denoted by $V_1,V_2,V_3$, have
primitive Pl\"ucker vectors $w_1,w_2,w_3$ satisfying
$$
\|w_i-\pi_cw_i\|\le\varepsilon
\qquad
(1\le i\le3),
$$
or the analogous inequalities with $\pi_r$. We treat the first case; the
second is identical.

Put $\widehat w_i=w_i/\|w_i\|$. Since $\wedge^2\Lambda$ is a fixed lattice,
$\|w_i\|\gg1$. We claim that there is a unit Pl\"ucker vector
$z_i\in\mathcal M_c$ such that
\begin{equation}\label{eq:near-exact-ruling}
\|\widehat w_i-z_i\|\ll\varepsilon.
\end{equation}
To see this, let
$
e_1=E_{11}, e_2=E_{12},
e_3=E_{21}, e_4=E_{22},
e_{ij}=e_i\wedge e_j,
$
and write
$$
\widehat w_i
=
\alpha e_{12}
+\beta(e_{14}-e_{23})
+\gamma e_{34}
+\rho e_{13}
+\sigma(e_{14}+e_{23})
+\tau e_{24}.
$$
The Pl\"ucker relation gives
$$
\alpha\gamma-\beta^2=\rho\tau-\sigma^2.
$$
The last three coefficients are $O(\varepsilon)$, and hence
$|\alpha\gamma-\beta^2|\ll\varepsilon^2$. Moreover, one of
$|\alpha|,|\beta|,|\gamma|$ is bounded below. If $|\alpha|\gg1$, replace
$\gamma$ by $\beta^2/\alpha$; if $|\gamma|\gg1$, replace $\alpha$ by
$\beta^2/\gamma$. In the remaining case, $|\beta|\gg1$ and
$\alpha\gamma>0$, so we replace $\beta$ by
$\operatorname{sgn}(\beta)\sqrt{\alpha\gamma}$. In each case the change is
$O(\varepsilon^2)$, and the resulting nonzero vector in $\mathcal M_c$ is
a Pl\"ucker vector. Normalizing it proves
\eqref{eq:near-exact-ruling}.

We next show that the planes $V_i$ are pairwise transverse. In the integral
coordinates induced by $\mathcal B$, the vectors $w_i$ are primitive
integral vectors of norm $O(R)$. Since $V_i\ne V_j$, the vectors $w_i$ and
$w_j$ are not proportional. Some $2\times2$ minor formed from their
coordinates is therefore a nonzero integer, and hence
\begin{equation}\label{eq:plucker-separation}
\min_{\pm}\|\widehat w_i\pm\widehat w_j\|
\gg R^{-2}
\qquad
(i\ne j).
\end{equation}
By \eqref{eq:near-exact-ruling}, the same lower bound, up to a fixed
constant, holds with $z_i,z_j$ in place of
$\widehat w_i,\widehat w_j$.

Every unit Pl\"ucker vector in $\mathcal M_c$ has, up to sign, the form
$$
z(s,t)
=
(se_1+te_3)\wedge(se_2+te_4),
\qquad
s^2+t^2=1.
$$
A direct calculation gives
$$
|z(s,t)\wedge z(s',t')|
=
(st'-ts')^2
=
\frac12\min_{\pm}\|z(s,t)\pm z(s',t')\|^2.
$$
It follows that $|z_i\wedge z_j|\gg R^{-4}$. Since
$\varepsilon\ll R^{-18}$, equation \eqref{eq:near-exact-ruling} now implies
that $w_i\wedge w_j\ne0$. Thus
\begin{equation}\label{eq:pairwise-transverse}
V_i\cap V_j=\{0\}
\qquad
(i\ne j).
\end{equation}

Let $g_{\mathcal B}:\mathbb R^4\to\M_2(\mathbb R)$ send the standard basis
to $\mathcal B$. Pull back the planes $V_i$ and the planes represented by
$z_i$ under $g_{\mathcal B}$, and retain the same notation.  In these
coordinates, $\Lambda=\mathbb Z^4$ and
$$
Q:=\det\circ g_{\mathcal B}=Q_{\Lambda,\mathcal B}.
$$
By Minkowski's second theorem, for each $i$ there are linearly independent
vectors $u_i,v_i\in V_i\cap\mathbb Z^4$ satisfying
$$
\|u_i\|,\|v_i\|\ll R.
$$
Let $L$ be the $9\times10$ integer matrix representing the nine linear
functionals
$$
P\mapsto P(u_i),\quad
P(v_i),\quad
P(u_i+v_i)-P(u_i)-P(v_i),
\qquad
1\le i\le3,
$$
on the ten-dimensional space $\operatorname{Sym}^2((\mathbb R^4)^*)$ of quadratic forms on $\mathbb R^4$. Then
$$
\ker L
=
\{P:P|_{V_1}=P|_{V_2}=P|_{V_3}=0\},
$$
and every entry of $L$ is $O(R^2)$. We claim that $\ker L$ is one-dimensional. By
\eqref{eq:pairwise-transverse}, we have
$\mathbb R^4=V_1\oplus V_2$, and $V_3$ is the graph of an invertible linear
map $T:V_1\to V_2$. A quadratic form vanishing on both $V_1$ and $V_2$ has
the form
$$
P(x+y)=2\beta(x,y)
\qquad
(x\in V_1,\ y\in V_2)
$$
for some bilinear form $\beta:V_1\times V_2\to\mathbb R$. It also vanishes
on $V_3$ precisely when $\beta(x,Tx)=0$ for every $x\in V_1$, or
equivalently when the form
$$
(x,y)\mapsto\beta(x,Ty)
$$
is alternating. Since the space of alternating forms on the
two-dimensional space $V_1$ is one-dimensional, $\dim\ker L=1$.
The rational line $\ker L$ is generated by a nondegenerate rational
split form. Let $Q_0$ be its primitive integral generator.
 Its
coefficients are, up to a common divisor, the signed $9\times9$ minors of
$L$. Since the entries of $L$ are $O(R^2)$,
\begin{equation}\label{eq:Q0-height}
\operatorname{ht}(Q_0)\ll R^{18}.
\end{equation}

It remains to compare $Q_0$ with $Q$. Let $W_i$ be the pulled-back
isotropic plane represented by $z_i$. The distance between the orthogonal
projections onto two planes is controlled by the distance between their
normalized Pl\"ucker vectors. Hence, if $P_U$ denotes orthogonal projection
onto $U$,
$$
\|P_{V_i}-P_{W_i}\|_{\operatorname{op}}
\ll
\|\widehat w_i-z_i\|
\ll\varepsilon.
$$
Since $Q$ vanishes on $W_i$, it follows that
$
\left|\frac{Q(x)}{\|Q\|}\right|
\ll
\varepsilon\|x\|^2$ for $x\in V_i$.
Applying this estimate to $u_i$, $v_i$, and $u_i+v_i$ gives
\begin{equation}\label{eq:LQ-small}
\left\|
L\left(\tfrac{Q}{\|Q\|}\right)
\right\|
\ll
R^2\varepsilon.
\end{equation}

Let $\sigma_1\ge\cdots\ge\sigma_9>0$ be the nonzero singular values of
$L$. For each set $I$ of nine columns, let $L_I$ be the corresponding
$9\times9$ submatrix. The Cauchy--Binet formula gives
$$
\det(LL^{\mathsf T})
=
\sum_{|I|=9}\det(L_I)^2.
$$
Since $L$ has rank $9$, one of these minors is a nonzero integer. Therefore
$$
1
\le
\det(LL^{\mathsf T})
=
\prod_{j=1}^9\sigma_j^2.
$$
As $\sigma_j\ll R^2$, we obtain $\sigma_9\gg R^{-16}$. It follows from
\eqref{eq:LQ-small} that
$$
\operatorname{dist}
\left(
\tfrac{Q}{\|Q\|},\ker L
\right)
\le
\sigma_9^{-1}
\left\|
L\left(\tfrac{Q}{\|Q\|}\right)
\right\|
\ll
R^{18}\varepsilon.
$$
Since $\ker L=\mathbb RQ_0$, this gives
$
\operatorname{dist}([Q],[Q_0])
\ll
R^{18}\varepsilon.
$
Together with \eqref{eq:Q0-height}, this proves the lemma.
\end{proof}

The reconstruction lemma converts the presence of several nearly isotropic
rational planes into an exceptionally good approximation of the determinant
form by a rational split form. The Diophantine condition therefore gives
the following uniform bound.

\begin{lem}[Finite-scale consequence of the Diophantine condition]
\label{lem:split-diophantine-implies-fs}
If $\Lambda$ is Diophantine, then there exist $0<\eta<1$ and $M>1$ such that
\begin{equation}\label{fs}
\#\mathcal Q_\Lambda(R,\eta R^{-M})\le4
\qquad
(R\ge1).
\end{equation}
\end{lem}

\begin{proof}
Fix a basis $\mathcal B$ and write $Q=Q_{\Lambda,\mathcal B}$. Suppose that
$\Lambda$ is $(\eta_0,M_0)$-split-Diophantine, and let $C$ be the constant
in \Cref{lem:d2-reconstruction-from-quasinull}. Choose
$M>18(M_0+1),$
and then choose $\eta>0$ sufficiently small that
$$
C\eta R^{18-M}
<
\eta_0C^{-M_0}R^{-18M_0}
\qquad
(R\ge1).
$$
If \eqref{fs} failed, then
\Cref{lem:d2-reconstruction-from-quasinull}, applied with
$\varepsilon=\eta R^{-M}$, would give a rational split form $Q_0$ satisfying
$
\operatorname{ht}(Q_0)\le CR^{18}
$
and
$$
\operatorname{dist}([Q],[Q_0])
\le
C\eta R^{18-M}
<
\eta_0C^{-M_0}R^{-18M_0}
\le
\eta_0\operatorname{ht}(Q_0)^{-M_0}.
$$
This contradicts the split-Diophantine condition.
\end{proof}

We now define the set of degree-$2$ vectors that will be omitted from the
modified height. For fixed $0<\eta<1$ and $M>1$, set
\begin{equation}\label{eq:d2-quasinull-cone}
\mathcal H_{\eta,M}
=
\left\{
0\ne w\in\wedge^2\M_2(\mathbb R):
\min\{\|w-\pi_cw\|,\|w-\pi_rw\|\}
\le
\eta\|w\|^{-M}
\right\}.
\end{equation}
We call $\mathcal H_{\eta,M}$ the \emph{quasi-null set}. A
$\Delta$-rational two-plane is quasi-null precisely when its primitive
Pl\"ucker vector belongs to $\mathcal H_{\eta,M}$.

\begin{defn}[Modified height]\label{def:modified-height-d2}
For $h\in H$ and $\Delta\in X$, define
$$
\widehat\alpha_{\eta,M}(h;\Delta)
=
\max\left\{
1,
\alpha_1(h\Delta),
\alpha_3(h\Delta),
\sup_{\substack{V\text{ a $\Delta$-rational two-plane}\\
\mathsf w_{\Delta,V}\notin\mathcal H_{\eta,M}}}
\|h\mathsf w_{\Delta,V}\|^{-1}
\right\}.
$$
The supremum over the empty set is understood to be zero.
\end{defn}

\section{Local contraction estimates and the critical weight}
\label{sec:local-d2}
We establish local contraction estimates for the action of the expanding
horospherical subgroup on the exterior powers of $\M_2(\mathbb R)$. In
degrees $1$ and $3$, we use exponents $1<\beta<2$. Degree $2$ is
critical: contraction holds for $0<\beta<1$, while the endpoint $\beta=1$
gives only a uniform bound. At this endpoint, the contraction becomes
weaker as a Pl\"ucker vector approaches $\mathcal M_c$ or $\mathcal M_r$.
We therefore introduce a weight measuring the distances from these two
summands.

We recall only the one-factor restrictions of the decompositions
\eqref{decom}--\eqref{decom2}. Under the left $\SL_2(\mathbb R)$-factor,
$\wedge^1\M_2(\mathbb R)$ and $\wedge^3\M_2(\mathbb R)$ restrict to
$\mathsf C^{\oplus2}$ and $(\mathsf C^*)^{\oplus2}$, respectively;
$\mathcal M_c$ restricts to $\operatorname{Sym}^2(\mathsf C)$, while
$\mathcal M_r$ is trivial. The analogous statements hold for the right
factor, with $\mathsf C$ and $\mathsf R$, and $\mathcal M_c$ and
$\mathcal M_r$, interchanged.

We begin with an elementary sublevel estimate. 
\begin{lem}[Affine-linear sublevel estimate]
\label{lem:linear-sublevel}
Let $E$ be a finite-dimensional Euclidean space, let $x,y\in E$, and let
$I\subset\mathbb R$ be a bounded interval. If $(x,y)\ne(0,0)$, then
$$
\Leb\{\xi\in I:
\|y+\xi x\|\le\varepsilon\max\{\|x\|,\|y\|\}\}
\le C_I\varepsilon
\qquad
(0<\varepsilon<1),
$$
where $C_I>0$ depends only on $I$.
\end{lem}

\begin{proof}
If $x=0$, then $y\ne0$, and the set is empty because
$\varepsilon<1$. Suppose that $x\ne0$. Dividing $x$ and $y$ by
$\|x\|$, we may assume that $\|x\|=1$. Write
$
y=ax+y^\perp$ where
$
y^\perp\perp x.
$
The assertion is immediate when $\varepsilon\ge1/4$, after increasing
$C_I$, so assume that $\varepsilon<1/4$. If the set is nonempty, then for some $\xi_0\in I$,
$$
|a+\xi_0|,\ \|y^\perp\|
\le
\varepsilon\max\{1,\|y\|\}.
$$
If $\|y\|>1$, then
$
a^2
=
\|y\|^2-\|y^\perp\|^2
\ge
(1-\varepsilon^2)\|y\|^2,
$
so $|a|\ge3\|y\|/4$. Together with
$|a+\xi_0|\le\varepsilon\|y\|$ and the boundedness of $I$, this implies
that $\|y\|\ll_I1$. The same conclusion is immediate when $\|y\|\le1$.
Consequently, every $\xi$ in the set satisfies
$|a+\xi|\ll_I\varepsilon$. The set of such $\xi\in I$ has length
$O_I(\varepsilon)$.
\end{proof}

\begin{prop}[One-factor local estimates]
\label{prop:local-contraction-d2}
Let $I\subset\mathbb R$ be a bounded interval.
\begin{enumerate}[label=\textnormal{(\roman*)}]
\item
If $V$ is a finite direct sum of copies of the standard
$\SL_2(\mathbb R)$-module $\mathbb R^2$ or its dual, then, for every
$1<\beta<2$, there is a constant $C_{I,\beta,V}>0$ such that for all $t\ge0$ and $0\ne v\in V$,
$$
\int_I\|b_tu_\xi v\|^{-\beta}\,d\xi
\le
C_{I,\beta,V}e^{-(2-\beta)t}\|v\|^{-\beta}.
$$

\item
If $V$ is a finite direct sum of copies of
$\operatorname{Sym}^2(\mathbb R^2)$ or its dual, then, for every
$0<\beta\le1$, there is a constant $C_{I,\beta,V}>0$ such that for all $t\ge0$ and $0\ne v\in V$,
$$
\int_I\|b_tu_\xi v\|^{-\beta}\,d\xi
\le
C_{I,\beta,V}e^{-c_\beta t}\|v\|^{-\beta},
\qquad
c_\beta=\min\{\beta,1-\beta\}.
$$
\end{enumerate}
\end{prop}

\begin{proof}[Sketch of proof]
This is a special case of
\cite[Lemma~5.1]{eskin-margulis-mozes:1998}. In the present setting, it
can also be proved by elementary calculations; we include a sketch.

For a finite direct sum, choose a summand whose component has norm
comparable to $\|v\|$. Since the norm of the full image controls the norm
of each component, it suffices, after changing the implicit constant, to
prove the estimate for a single copy. We may also choose any convenient
equivalent norm.
For the standard module, normalize $v=(x,y)$ by $\|v\|=1$. Then
$$
b_tu_\xi v=(e^{-t}x,e^t(y+\xi x)).
$$
Write $I\subset[-R,R]$. If $|x|\le(2R+2)^{-1}$, then
$|y+\xi x|\gg_I1$ on $I$, and the integral is
$O_{I,\beta}(e^{-\beta t})$. Since $\beta>1$, this is
$O_{I,\beta}(e^{-(2-\beta)t})$.

If $|x|>(2R+2)^{-1}$, the change of variables
$r=e^{2t}(\xi+y/x)$ gives
$$
\begin{aligned}
\int_I\|b_tu_\xi v\|^{-\beta}\,d\xi
&\ll_{I,\beta}
|x|^{-\beta}e^{-(2-\beta)t}
\int_{\mathbb R}(1+r^2)^{-\beta/2}\,dr \ll_{I,\beta}
e^{-(2-\beta)t}.
\end{aligned}
$$
The last integral is finite because $\beta>1$. This proves
\textnormal{(i)} after rescaling.

For the symmetric-square module, use the basis
$Z_1^2,Z_1Z_2,Z_2^2$, of weights $-2,0,2$, and write
$$
v=aZ_1^2+bZ_1Z_2+cZ_2^2,
\qquad
\|v\|=1.
$$
Up to fixed normalization, the coordinates of $u_\xi v$ are
$a, b+2a\xi, c+b\xi+a\xi^2$.
Thus
$$
\|b_tu_\xi v\|^2
\asymp
e^{-4t}|a|^2
+
|b+2a\xi|^2
+
e^{4t}|c+b\xi+a\xi^2|^2.
$$
The estimate \Cref{lem:linear-sublevel}, applied to the
middle coordinate $b+2a\xi$, provides the uniform sublevel control
in the case decomposition of
\cite[Lemma~5.1]{eskin-margulis-mozes:1998}. The contribution with the
slowest decay is reduced, after translating the variable, to the model
integral
$
\int_I
\left(e^{-4t}+\xi^2+e^{4t}\xi^4\right)^{-\beta/2}\,d\xi.
$
A direct decomposition at $|\xi|=e^{-2t}$ gives
$$
\int_I
\left(e^{-4t}+\xi^2+e^{4t}\xi^4\right)^{-\beta/2}\,d\xi
\ll_{I,\beta}
\begin{cases}
e^{-2\beta t},&0<\beta<1/2,\\
(1+t)e^{-t},&\beta=1/2,\\
e^{-2(1-\beta)t},&1/2<\beta<1,\\
1,&\beta=1.
\end{cases}
$$
In particular, this is $O_{I,\beta}(e^{-c_\beta t})$ for
$0<\beta<1$, where $c_\beta=\min\{\beta,1-\beta\}$, and is uniformly
bounded at $\beta=1$. This proves \textnormal{(ii)} after rescaling. The dual standard and dual symmetric-square modules are isomorphic to
their respective original modules, so the same estimates apply to them.
\end{proof}

We now pass from the one-factor estimates to the product group.

\begin{prop}[Product local estimates]
\label{prop:product-local-d2}
Let $\Omega\subset N$ be bounded.
\begin{enumerate}[label=\textnormal{(\roman*)}]
\item
If $V$ is either $\wedge^1\M_2(\mathbb R)$ or
$\wedge^3\M_2(\mathbb R)$, then, for every $1<\beta<2$, there is a
constant $C_{\Omega,\beta}>0$ such that for all $t\ge0$ and $0\ne v\in V$,
$$
\int_\Omega\|a_tnv\|^{-\beta}\,dn
\le
C_{\Omega,\beta}e^{-(2-\beta)t}\|v\|^{-\beta}.
$$

\item
For every $0<\beta\le1$, there is a constant
$C_{\Omega,\beta}>0$ such that, for every
$0\ne v\in\wedge^2\M_2(\mathbb R)$,
\begin{equation}\label{eq:d2-subcritical-degree-two-local}
\int_\Omega\|a_tnv\|^{-\beta}\,dn
\le
C_{\Omega,\beta}e^{-c_\beta t}\|v\|^{-\beta},
\qquad
c_\beta=\min\{\beta,1-\beta\}.
\end{equation}
\end{enumerate}
\end{prop}

\begin{proof}
For the first claim, choose bounded intervals $I_1,I_2$ such that
$\Omega\subset I_1\times I_2$. Applying
\Cref{prop:local-contraction-d2}\textnormal{(i)} first in $\xi_1$ and
then in $\xi_2$ gives
$$
\int_{I_1\times I_2}
\|a_tn_{\xi_1,\xi_2}v\|^{-\beta}\,d\xi_1\,d\xi_2
\ll_{\Omega,\beta}
e^{-2(2-\beta)t}\|v\|^{-\beta}.
$$
Since $t\ge0$, this implies the stated estimate.
 For the second, write $v=v_c+v_r$ according to the orthogonal decomposition \eqref{decom}. If $\|v_c\|\ge\|v_r\|$, then
$$
\|a_tn v\|^{-\beta}\le \|b_tu_{\xi_1}v_c\|^{-\beta},
$$
and \Cref{prop:local-contraction-d2}$(ii)$ applied in $\xi_1$ gives the required estimate. If $\|v_r\|\ge\|v_c\|$, apply the same argument in $\xi_2$. Since $\max(\|v_c\|,\|v_r\|)\gg\|v\|$, both cases give \eqref{eq:d2-subcritical-degree-two-local}.
\end{proof}

The endpoint estimate in degree $2$ does not provide contraction near
$\mathcal M_c$ or $\mathcal M_r$. We therefore introduce a weight that
penalizes proximity to these two summands. For $0<\theta<1/4$ and
$0\ne w\in\wedge^2\M_2(\mathbb R)$ with
$w\notin\mathcal M_c\cup\mathcal M_r$, define
\begin{equation}\label{eq:d2-local-weight}
\phi_\theta(w)
=
\|w\|^{-1+4\theta}
\|w-\pi_cw\|^{-2\theta}
\|w-\pi_rw\|^{-2\theta}.
\end{equation}
For $w\in\mathcal M_c\cup\mathcal M_r$, set
$\phi_\theta(w)=+\infty$. 

We first record the elementary distortion estimates used to globalize the
local bounds.

\begin{lem}[Log-Lipschitz estimates]
\label{lem:d2-log-lipschitz}
For all $s\ge0$, $n\in B_N(1)$, $1\le i\le3$, and
$v\in\wedge^i\M_2(\mathbb R)$,
\begin{equation}\label{eq:d2-log-lipschitz-wedge}
e^{-2i(s+1)}\|v\|
\le
\|a_snv\|
\le
e^{2i(s+1)}\|v\|.
\end{equation}
In particular, the exponent is at most $6$ in every exterior degree. In
degree $2$, whenever both sides are finite,
\begin{equation}\label{eq:d2-log-lipschitz-phi}
\phi_\theta(a_snv)
\le
e^{4(s+1)}\phi_\theta(v).
\end{equation}
\end{lem}

\begin{proof}
For $|\xi|\le1$, both $b_su_\xi$ and its inverse have operator norm at
most $e^{s+1}$ on $\mathbb R^2$. Hence $a_sn$, acting on
$\M_2(\mathbb R)\simeq\mathbb R^2\otimes\mathbb R^2$, has operator norm
and inverse operator norm at most $e^{2(s+1)}$. Passing to the $i$th exterior
power proves \eqref{eq:d2-log-lipschitz-wedge}. The projections $\pi_c$ and $\pi_r$ commute with the $H$-action, so the
degree-$2$ estimate applies to $v$, $v-\pi_cv$, and $v-\pi_rv$. Since
$0<\theta<1/4$,
$$
\begin{aligned}
\phi_\theta(a_snv)
&\le
e^{4(1-4\theta)(s+1)}
e^{8\theta(s+1)}
e^{8\theta(s+1)}
\phi_\theta(v)=
e^{4(s+1)}\phi_\theta(v).
\end{aligned}
$$
This proves \eqref{eq:d2-log-lipschitz-phi}.
\end{proof}

The following is the main local input in the critical
degree.

\begin{lem}[Weighted local estimate in degree $2$]
\label{lem:weighted-local-d2}
For every $0<\theta<1/10$, there is a constant $C_\theta>0$ such that,
for every $s\ge0$ and every
$0\ne w\in\wedge^2\M_2(\mathbb R)$ with
$w\notin\mathcal M_c\cup\mathcal M_r$,
$$
\int_{B_N(1)}
\phi_\theta(a_snw)^{1+\theta}\,dn
\le
C_\theta e^{-\theta s}\phi_\theta(w)^{1+\theta}.
$$
\end{lem}

\begin{proof}
Write
$$
w=w_c+w_r,
\qquad
w_c:=\pi_cw\in\cM_c,
\qquad
w_r:=\pi_rw\in\cM_r.
$$
Since $w\notin\cM_c\cup\cM_r$, both $w_c$ and $w_r$ are nonzero. For
$n=n_{\xi_1,\xi_2}$, the two factors act separately
and hence
$$
a_snw-\pi_c(a_snw)=b_su_{\xi_2}w_r,
\qquad
a_snw-\pi_r(a_snw)=b_su_{\xi_1}w_c.
$$
Because $\cM_c$ and $\cM_r$ are fixed complementary subspaces,
$\|w\|\asymp\max\{\|w_c\|,\|w_r\|\}$. Suppose first that
$\|w_c\|\ge\|w_r\|$. Then
$$
\phi_\theta(w)^{1+\theta}
\asymp_\theta
\|w_c\|^{-(1+\theta)(1-2\theta)}\|w_r\|^{-2(1+\theta)\theta}.
$$
Also, since $\|a_snw\|\ge c\|b_su_{\xi_1}w_c\|$, we have
$$
\phi_\theta(a_snw)^{1+\theta}
\ll_\theta
\|b_su_{\xi_1}w_c\|^{-(1+\theta)(1-2\theta)}
\|b_su_{\xi_2}w_r\|^{-2(1+\theta)\theta}.
$$
For $0<\theta<1/10$, both exponents
$(1+\theta)(1-2\theta)$ and $2(1+\theta)\theta$ are strictly smaller than $1$. Moreover their
$c_\beta$-rates from \Cref{prop:local-contraction-d2}$(ii)$ add up to at
least $\theta$. Applying that estimate in the two variables gives
$$
\begin{aligned}
\int_{B_N(1)}\phi_\theta(a_snw)^{1+\theta}\,dn
&\ll e^{-\theta s}\|w_c\|^{-(1+\theta)(1-2\theta)}
\|w_r\|^{-2(1+\theta)\theta}\ll  e^{-\theta s}\phi_\theta(w)^{1+\theta}.
\end{aligned}
$$
The case $\|w_r\|\ge\|w_c\|$ is identical, with $w_c$ and $w_r$ interchanged.
\end{proof}

We conclude by recording the exponents used later in degrees $1$ and $3$.

\begin{lem}[Local estimates in degrees $1$ and $3$]
\label{lem:local-dim-1-3-d2}
For every $0<\theta<1/100$, there is a constant $C_\theta>0$ such that,
for every $s\ge0$, $h\in H$, $\Delta\in X$, and every
$\Delta$-rational subspace $V$ of dimension $1$ or $3$,
$$
\int_{B_N(1)}
\|a_snh\mathsf w_{\Delta,V}\|^{-(1+\sqrt\theta)}\,dn
\le
C_\theta e^{-(1-\sqrt\theta)s}
\|h\mathsf w_{\Delta,V}\|^{-(1+\sqrt\theta)}.
$$
\end{lem}

\begin{proof}
In degrees $1$ and $3$, the relevant representations are
$\mathsf C\otimes\mathsf R$ and
$\mathsf C^*\otimes\mathsf R^*$, respectively. Apply
\Cref{prop:product-local-d2}\textnormal{(i)} with
$\beta=1+\sqrt\theta$.
\end{proof}
All the estimates in this section remain valid with $B_N(1)$ replaced by
$B_N(R)$ for any fixed $R>0$, with constants that may depend on $R$.

\section{Global height functions and contraction estimates}
\label{s:globalheight}

We combine the local estimates from
\Cref{sec:local-d2} with the intersection-sum inequality to obtain a
global contraction estimate for the auxiliary modified height. In the
present $2\times2$ setting, degree $2$ is the only critical exterior
degree, whereas several critical degrees occur in the higher-dimensional
argument of \cite{KimOh_det}. Consequently, the resulting Margulis
inequality has no logarithmic factor, and the standard intersection-sum
inequality of \cite{eskin-margulis-mozes:1998} suffices in place of the
Benoist--Quint mother inequality used in the higher-dimensional argument
of \cite{KimOh_det}. We isolate the remaining obstruction to contraction
in the exceptional set $\mathcal E_{s,\eta,M,M'}$ introduced below.

Fix $0<\eta<1$, $M>1$, and $0<\theta<10^{-4}$.

\begin{defn}[Auxiliary modified height]\label{def:auxiliary-height-d2}
For \(h\in H\) and \(\Delta\in X\), define
\begin{equation}\label{eq:d2-global-tilde-height}
\widetilde\alpha_{\eta,M,\theta}(h;\Delta):=\max\left\{
\begin{array}{l}
1,\ \alpha_1(h\Delta)^{1+\sqrt\theta},\ \alpha_3(h\Delta)^{1+\sqrt\theta},\\[1mm]
\displaystyle \sup_{\substack{V\text{ \(\Delta\)-rational two-plane}\\ \mathsf w_{\Delta,V}\notin\mathcal H_{\eta,M},\ 0<\|h\mathsf w_{\Delta,V}\|\le1}}
\phi_\theta(h\mathsf w_{\Delta,V})^{1+\theta},\\[4mm]
\displaystyle \sup_{\substack{V\text{ \(\Delta\)-rational two-plane}\\ \mathsf w_{\Delta,V}\in\mathcal H_{\eta,M},\ 0<\|h\mathsf w_{\Delta,V}\|\le1}}
\|h\mathsf w_{\Delta,V}\|^{-(1-\theta)}
\end{array}\right\}.
\end{equation}
The supremum over an empty set is zero.
\end{defn}

The sets indexing the two suprema in
\eqref{eq:d2-global-tilde-height} are finite: the condition
$\|h\mathsf w_{\Delta,V}\|\le1$ restricts the primitive Pl\"ucker vectors
to a bounded subset of a discrete set. Moreover, since
$\|w-\pi_cw\|,\|w-\pi_rw\|\ll\|w\|$, for $0<\|w\|\le1$ we have
$\phi_\theta(w)\gg\|w\|^{-1}$. It follows that
\begin{equation}\label{eq:auxiliary-comparisons}
\widehat\alpha_{\eta,M}(h;\Delta)^{1+\theta}
\ll_\theta
\widetilde\alpha_{\eta,M,\theta}(h;\Delta),
\qquad
\alpha(h\Delta)^{1-\theta}
\ll_\theta
\widetilde\alpha_{\eta,M,\theta}(h;\Delta).
\end{equation}

We shall repeatedly use the following form of Young's inequality:
\begin{equation}\label{eq:scaled-young-d2}
uA^q
\le
\rho A
+
C_q\rho^{-q/(1-q)}u^{1/(1-q)}
\qquad
(0<q<1,\ u,A,\rho>0).
\end{equation}
We also use the intersection-sum inequality
\cite[Lemma~5.6]{eskin-margulis-mozes:1998}:
\begin{equation}\label{eq:intersection-sum-d2}
\|\mathsf w_{\Delta,V\cap W}\|\,
\|\mathsf w_{\Delta,V+W}\|
\ll
\|\mathsf w_{\Delta,V}\|\,
\|\mathsf w_{\Delta,W}\|,
\end{equation}
where the covolumes in dimensions $0$ and $4$ are interpreted as $1$.

\begin{lem}[Global estimates in dimensions $1$ and $3$]
\label{lem:d2-noncritical-height-contribution}
There is a constant $C_\theta>0$ such that, for every $s\ge0$, $h\in H$,
and $\Delta\in X$,
\begin{equation}\label{eq:d2-noncritical-tilde-bound}
\begin{aligned}
&\int_{B_N(1)}
\left(
\alpha_1(a_snh\Delta)^{1+\sqrt\theta}
+
\alpha_3(a_snh\Delta)^{1+\sqrt\theta}
\right)\,dn\\
&\qquad\le
C_\theta e^{-(1-\sqrt\theta)s}
\widetilde\alpha_{\eta,M,\theta}(h;\Delta)
+
C_\theta e^{C_\theta s}.
\end{aligned}
\end{equation}
\end{lem}

\begin{proof}
Fix $i\in\{1,3\}$, and let $\mathcal P$ be the finite set of
$i$-dimensional $\Delta$-rational subspaces $V$ satisfying
$$
\|h\mathsf w_{\Delta,V}\|\le e^{2i(s+1)}.
$$
If $\|a_snh\mathsf w_{\Delta,V}\|<1$, then
\eqref{eq:d2-log-lipschitz-wedge} implies that $V\in\mathcal P$.

For distinct $V,W\in\mathcal P$, the intersection-sum inequality
\eqref{eq:intersection-sum-d2}, together with
\eqref{eq:auxiliary-comparisons}, gives
$$
\begin{aligned}
&\|h\mathsf w_{\Delta,V}\|^{-(1+\sqrt\theta)}
\|h\mathsf w_{\Delta,W}\|^{-(1+\sqrt\theta)}\ll
\widetilde\alpha_{\eta,M,\theta}(h;\Delta)^{
(1+\sqrt\theta)/(1-\theta)}.
\end{aligned}
$$
For $i=1$, the only nontrivial factor on the left-hand side of
\eqref{eq:intersection-sum-d2} is associated with $V+W$; for $i=3$, it
is associated with $V\cap W$.

If $\mathcal P\ne\emptyset$, choose $V_0\in\mathcal P$ maximizing
$\|h\mathsf w_{\Delta,V}\|^{-(1+\sqrt\theta)}$. Then, for
$V\ne V_0$,
$$
\|h\mathsf w_{\Delta,V}\|^{-(1+\sqrt\theta)}
\ll
\widetilde\alpha_{\eta,M,\theta}(h;\Delta)^{
(1+\sqrt\theta)/(2(1-\theta))}.
$$
Since
$
\frac{1+\sqrt\theta}{2(1-\theta)}
=
\frac{1}{2(1-\sqrt\theta)}
<1,
$
we obtain, pointwise in $n$,
$$
\begin{aligned}
\alpha_i(a_snh\Delta)^{1+\sqrt\theta}
\ll\;&
1+
\|a_snh\mathsf w_{\Delta,V_0}\|^{-(1+\sqrt\theta)}\\
&+
e^{2i(1+\sqrt\theta)(s+1)}
\widetilde\alpha_{\eta,M,\theta}(h;\Delta)^{
(1+\sqrt\theta)/(2(1-\theta))},
\end{aligned}
$$
where we omit the middle term if $\mathcal P$ is empty. Integrating the
middle term using \Cref{lem:local-dim-1-3-d2}, and applying
\eqref{eq:scaled-young-d2} to the last term with
$\rho=e^{-(1-\sqrt\theta)s}$, proves the required estimate for $i$.
Summing over $i=1,3$ completes the proof.
\end{proof}

\subsection{The exceptional set}

The auxiliary height can fail to contract when a non-quasi-null rational
two-plane of controlled covolume is carried very close to one of the
exceptional summands. We now isolate this possibility.

For $M'\ge1$, define
\begin{equation}\label{eq:d2-exceptional-scale}
\varepsilon_{s,M'}(h;\Delta)
:=
\begin{cases}
e^{-M's},
&
\alpha(h\Delta)\le e^{5(s+1)},\\
\alpha(h\Delta)^{-M'},
&
\alpha(h\Delta)>e^{5(s+1)}.
\end{cases}
\end{equation}
Let $\mathcal E_{s,\eta,M,M'}$ be the set of pairs
$(h,\Delta)\in H\times X$ for which there is a $\Delta$-rational two-plane
$V$, with
$
w=\mathsf w_{\Delta,V}\notin\mathcal H_{\eta,M},
$
such that
$$
0<\|hw\|\le e^{4(s+1)},
\qquad
\min\{\|hw-\pi_c(hw)\|,\|hw-\pi_r(hw)\|\}
\le
\varepsilon_{s,M'}(h;\Delta).
$$

\begin{lem}\label{lem:d2-quasinull-at-identity}
Let $\Delta\in X$, $0<\eta\le1$, and $M\ge1$. There is
$s_0=s_0(\Delta,\eta,M)$ such that
$(e,\Delta)\notin\mathcal E_{s,\eta,M,5M}$ for every $s\ge s_0$.
\end{lem}

\begin{proof}
For all sufficiently large $s$, we have
$\alpha(\Delta)\le e^{5(s+1)}$, so the first case in
\eqref{eq:d2-exceptional-scale} applies. If $(e,\Delta)$ were exceptional,
there would be some
$0\ne w\notin\mathcal H_{\eta,M}$ satisfying $\|w\|\le e^{4(s+1)}$ and
$$
\min\{\|w-\pi_cw\|,\|w-\pi_rw\|\}
\le
e^{-5Ms}
\le
\eta e^{-4M(s+1)}
\le
\eta\|w\|^{-M}
$$
for all sufficiently large $s$. This would imply
$w\in\mathcal H_{\eta,M}$, a contradiction.
\end{proof}

\subsection{One-step global contraction}

We now combine the local estimates with the preceding separation
arguments.

\begin{prop}[Global contraction outside the exceptional set]
\label{prop:globalcontraction}
There is a constant $C_\theta>1$ such that the following holds. Let
$0<\eta<1$, $M>1$, $M'\ge1$, and $s\ge0$, and suppose that
$M'\sqrt\theta\le1/20$. If
$(h,\Delta)\notin\mathcal E_{s,\eta,M,M'}$, then
\begin{equation}\label{eq:d2-globalcontraction}
\int_{B_N(1)}
\widetilde\alpha_{\eta,M,\theta}(a_snh;\Delta)\,dn
\le
C_\theta e^{-\theta s/2}
\widetilde\alpha_{\eta,M,\theta}(h;\Delta)
+
C_\theta e^{C_\theta s}.
\end{equation}
\end{prop}
\begin{proof}
The degree-$1$ and degree-$3$ terms are bounded by
\eqref{eq:d2-noncritical-tilde-bound}.

We first consider the quasi-null supremum in
\eqref{eq:d2-global-tilde-height}. Let $\mathcal P$ be the finite set of
quasi-null rational two-planes $V$ satisfying
$\|h\mathsf w_{\Delta,V}\|\le e^{4(s+1)}$. If
$\mathcal P\ne\emptyset$, choose $V_0\in\mathcal P$ maximizing
$\|h\mathsf w_{\Delta,V}\|^{-(1-\theta)}$.

For distinct $V,W\in\mathcal P$, the intersection-sum inequality gives
$$
\|h\mathsf w_{\Delta,V}\|^{-(1-\theta)}
\|h\mathsf w_{\Delta,W}\|^{-(1-\theta)}
\ll
\widetilde\alpha_{\eta,M,\theta}(h;\Delta)^{
2(1-\theta)/(1+\sqrt\theta)}.
$$
Hence, pointwise in $n$, the quasi-null supremum is at most
$$
1+
\|a_snh\mathsf w_{\Delta,V_0}\|^{-(1-\theta)}
+
e^{4(1-\theta)(s+1)}
\widetilde\alpha_{\eta,M,\theta}(h;\Delta)^{
(1-\theta)/(1+\sqrt\theta)},
$$
where the middle term is omitted if $\mathcal P$ is empty. Since
$
\frac{1-\theta}{1+\sqrt\theta}=1-\sqrt\theta<1,
$
we may integrate the middle term using
\eqref{eq:d2-subcritical-degree-two-local} with exponent $1-\theta$ and
apply \eqref{eq:scaled-young-d2} to the last term with
$\rho=e^{-\theta s/2}$. This gives
\begin{equation}\label{eq:quasinull-global-bound}
\begin{aligned}
&\int_{B_N(1)}
\sup_{\substack{V:\ \mathsf w_{\Delta,V}\in\mathcal H_{\eta,M}\\
0<\|a_snh\mathsf w_{\Delta,V}\|\le1}}
\|a_snh\mathsf w_{\Delta,V}\|^{-(1-\theta)}\,dn \ll
e^{-\theta s/2}
\widetilde\alpha_{\eta,M,\theta}(h;\Delta)+e^{C_\theta s}.
\end{aligned}
\end{equation}

It remains to bound the non-quasi-null supremum. Set
$$
\Phi_{\mathrm{nq}}(h;\Delta)
:=
\max\left\{
1,
\sup_{\substack{V:\ \mathsf w_{\Delta,V}\notin\mathcal H_{\eta,M}\\
0<\|h\mathsf w_{\Delta,V}\|\le1}}
\phi_\theta(h\mathsf w_{\Delta,V})^{1+\theta}
\right\},
$$
and let $\mathcal P_2$ be the finite collection of planes occurring in
this supremum for which
$$
\phi_\theta(h\mathsf w_{\Delta,V})^{1+\theta}
>
e^{-12(s+1)}\Phi_{\mathrm{nq}}(h;\Delta).
$$

Fix a $\Delta$-rational two-plane $V$, write
$w=\mathsf w_{\Delta,V}$, and suppose that
$$
w\notin\mathcal H_{\eta,M},
\qquad
0<\|a_snhw\|\le1.
$$
Then $\|hw\|\le e^{4(s+1)}$. If $\|hw\|\le1$ and
$V\notin\mathcal P_2$, the log-Lipschitz estimate
\eqref{eq:d2-log-lipschitz-phi} gives
$$
\phi_\theta(a_snhw)^{1+\theta}
\le
e^{-4(s+1)}\Phi_{\mathrm{nq}}(h;\Delta).
$$
If $1<\|hw\|\le e^{4(s+1)}$, the assumption that
$(h,\Delta)\notin\mathcal E_{s,\eta,M,M'}$ gives
$$
\phi_\theta(a_snhw)^{1+\theta}
\ll
e^{4(1+\theta)(s+1)}
\varepsilon_{s,M'}(h;\Delta)^{-4\theta(1+\theta)}.
$$

If $\varepsilon_{s,M'}(h;\Delta)=e^{-M's}$, this is
$O(e^{C_\theta s})$, since $M'\sqrt\theta\le1/20$. If instead
$\varepsilon_{s,M'}(h;\Delta)=\alpha(h\Delta)^{-M'}$, it is at most
$$
e^{4(1+\theta)(s+1)}
\widetilde\alpha_{\eta,M,\theta}(h;\Delta)^{
4M'\theta(1+\theta)/(1-\theta)}.
$$
The assumption $M'\sqrt\theta\le1/20$ gives
$$
\frac{4M'\theta(1+\theta)}{1-\theta}
\le
\frac{\sqrt\theta}{4}
<
\frac12.
$$
Since
$\Phi_{\mathrm{nq}}(h;\Delta)\le\widetilde\alpha_{\eta,M,\theta}(h;\Delta)$,
Young's inequality now gives
\begin{equation}\label{eq:outside-critical-competitors}
\begin{aligned}
&\int_{B_N(1)}
\sup_{\substack{V\notin\mathcal P_2:\\
\mathsf w_{\Delta,V}\notin\mathcal H_{\eta,M}\\
0<\|a_snh\mathsf w_{\Delta,V}\|\le1}}
\phi_\theta(a_snh\mathsf w_{\Delta,V})^{1+\theta}\,dn \ll
e^{-\theta s/2}
\widetilde\alpha_{\eta,M,\theta}(h;\Delta)+e^{C_\theta s}.
\end{aligned}
\end{equation}

\smallskip
\noindent\textbf{Case 1: $\#\mathcal P_2\le1$.}

In this case, \Cref{lem:weighted-local-d2} gives
$$
\int_{B_N(1)}
\sup_{V\in\mathcal P_2}
\phi_\theta(a_snh\mathsf w_{\Delta,V})^{1+\theta}\,dn
\ll
e^{-\theta s}
\widetilde\alpha_{\eta,M,\theta}(h;\Delta).
$$

\smallskip
\noindent\textbf{Case 2: $\#\mathcal P_2\ge2$.}

Suppose that $\mathcal P_2$ contains two distinct planes with Pl\"ucker
vectors $w_1$ and $w_2$. The definition of $\mathcal P_2$, exclusion from
the exceptional set, and \eqref{eq:intersection-sum-d2} imply
$$
\begin{aligned}
\Phi_{\mathrm{nq}}(h;\Delta)^2
\ll\;&
e^{24(s+1)}
\varepsilon_{s,M'}(h;\Delta)^{-8\theta(1+\theta)}
\max\{1,\alpha_1(h\Delta)\alpha_3(h\Delta)\}
^{(1-4\theta)(1+\theta)}.
\end{aligned}
$$
If $\varepsilon_{s,M'}(h;\Delta)=e^{-M's}$, then
$$
\Phi_{\mathrm{nq}}(h;\Delta)
\ll
e^{12(s+1)}
e^{4M'\theta(1+\theta)s}
\widetilde\alpha_{\eta,M,\theta}(h;\Delta)^{
(1-4\theta)(1+\theta)/(1+\sqrt\theta)}.
$$
If $\varepsilon_{s,M'}(h;\Delta)=\alpha(h\Delta)^{-M'}$, then
$$
\Phi_{\mathrm{nq}}(h;\Delta)
\ll
e^{12(s+1)}
\widetilde\alpha_{\eta,M,\theta}(h;\Delta)^{
\frac{(1-4\theta)(1+\theta)}{1+\sqrt\theta}
+
\frac{4M'\theta(1+\theta)}{1-\theta}
}.
$$
Under the assumptions $M'\sqrt\theta\le1/20$ and $\theta<10^{-4}$,
$$
\frac{(1-4\theta)(1+\theta)}{1+\sqrt\theta}
\le
1-\frac{\sqrt\theta}{2},
\qquad
\frac{4M'\theta(1+\theta)}{1-\theta}
\le
\frac{\sqrt\theta}{4}.
$$
Thus, in both cases, \Cref{lem:d2-log-lipschitz} bounds the contribution
of $\mathcal P_2$ by
$$
e^{C_\theta s}
\widetilde\alpha_{\eta,M,\theta}(h;\Delta)^{1-\sqrt\theta/4}.
$$
Applying \eqref{eq:scaled-young-d2} with
$\rho=e^{-\theta s/2}$ gives
$$
\sup_{V\in\mathcal P_2}
\phi_\theta(a_snh\mathsf w_{\Delta,V})^{1+\theta}
\ll
e^{-\theta s/2}
\widetilde\alpha_{\eta,M,\theta}(h;\Delta)+e^{C_\theta s}.
$$

Together with \eqref{eq:outside-critical-competitors}, this bounds the
non-quasi-null supremum. Combining this with
\eqref{eq:quasinull-global-bound} and the degree-$1$ and degree-$3$
estimate proves \eqref{eq:d2-globalcontraction}.
\end{proof}
\section{Avoidance estimates}
\label{s:avoidance}

We now show that the exceptional set introduced in the previous section is
visited by only a small set of parameters in the expanding horospherical
averages. The main input is the quantitative nondivergence theorem of
\cite{kleinbock:2008}. In the present $2\times2$ setting, degree $2$ is the
only critical exterior degree, and we write
$$
E:=\wedge^2\M_2(\mathbb R)
=
\mathcal M_c\oplus\mathcal M_r.
$$
The obstruction alternative in quantitative nondivergence produces a
rational subspace of $E$ lying close to one of the exceptional summands.
The Diophantine condition rules out this alternative by converting such a
subspace into a rational split form that approximates the determinant form
too closely.

Throughout this section, fix a lattice $\Delta\in X$ that is not
determinant-rational and satisfies \eqref{fs} for some $0<\eta<1$ and
$M>1$. Fix also a basis $\mathcal B$ for which
$Q_{\Delta,\mathcal B}$ is $(\eta_0,M_0)$-split-Diophantine. We write
$\Leb$ for coordinate Lebesgue measure on $N\simeq\mathbb R^2$.

If $W<E$ is a three-dimensional subspace rational with respect to
$\wedge^2\Delta$, let $\mathsf w_{\wedge^2\Delta,W}$ denote the primitive
Pl\"ucker vector of the lattice $W\cap\wedge^2\Delta$. For
three-dimensional subspaces $W_1,W_2<E$, choose unit vectors $z_i$
spanning $\wedge^3W_i$ and set
$$
\operatorname{dist}(W_1,W_2)
:=
\min_{\pm}\|z_1\pm z_2\|.
$$
This quantity is independent of the choices of $z_1$ and $z_2$.

The first lemma reconstructs a rational split quadratic form from a
rational three-dimensional subspace lying close to one of the exceptional
summands.

\begin{lem}[Reconstruction near an exceptional summand]
\label{lem:rational-ruling-reconstruction-d2}
There exist constants $c,C>0$, depending only on $\mathcal B$, with the
following property. Let $W<E$ be a three-dimensional subspace rational
with respect to $\wedge^2\Delta$, and suppose that $W$ contains a nonzero
vector of the form $u\wedge v$ with $u,v\in\Delta$. If
$\operatorname{dist}(W,\mathcal M_c)<c$, then there is a rational split
quadratic form $Q_W$ such that
$$
\operatorname{ht}(Q_W)
\le
C\|\mathsf w_{\wedge^2\Delta,W}\|^3,
\qquad
\operatorname{dist}([Q_{\Delta,\mathcal B}],[Q_W])
\le
C\operatorname{dist}(W,\mathcal M_c).
$$
The same statement holds with $\mathcal M_r$ in place of $\mathcal M_c$.
\end{lem}

\begin{proof}
Let $g_{\mathcal B}:\mathbb R^4\to\M_2(\mathbb R)$ send the standard basis
to $\mathcal B$, and put
$$
W'
=
(\wedge^2g_{\mathcal B}^{-1})W,
\qquad
W_0
=
(\wedge^2g_{\mathcal B}^{-1})\mathcal M_c.
$$
Since $g_{\mathcal B}$ is fixed, all norms and distances change by only
fixed factors.

Let $U<\wedge^2\mathbb R^4$ be a three-dimensional subspace with basis
$\omega_1,\omega_2,\omega_3$. For $x\in\mathbb R^4$, define the
$3\times4$ matrix $L_U(x)$ by
$$
x\wedge\omega_i\wedge y
=
(L_U(x)y)_i\,
e_1\wedge e_2\wedge e_3\wedge e_4
\qquad
(1\le i\le3).
$$
 Then
$L_U(x)x=0$. Let $\mathcal C_U(x)$ denote the signed cofactor vector of
$L_U(x)$.

In matrix-unit coordinates, a direct calculation gives
$$
\mathcal C_{\mathcal M_c}(x)
=
c_0(x_1x_4-x_2x_3)x
$$
for some $c_0\ne0$. More generally, if $g\in\GL_4(\mathbb R)$, then
$$
L_{(\wedge^2g)U}(gx)g
=
\det(g)L_U(x).
$$
Taking signed minors shows that
$\mathcal C_{(\wedge^2g)U}(gx)$ is a fixed nonzero scalar multiple of
$g\mathcal C_U(x)$. It follows that the quadratic form associated to
$(\wedge^2g)U$ is proportional to the form associated to $U$ composed
with $g^{-1}$. In particular, the form associated to $W_0$ is proportional
to $Q_{\Delta,\mathcal B}$.

For any $x$ satisfying $Q_{\Delta,\mathcal B}(x)\ne0$, the matrix
associated to $W_0$ has rank $3$. After decreasing $c$, the matrix
$L_{W'}(x)$ has rank $3$ on a nonempty open set. On this set, both $x$ and
$\mathcal C_{W'}(x)$ span $\ker L_{W'}(x)$, and hence
$$
x_i\mathcal C_{W',j}(x)-x_j\mathcal C_{W',i}(x)=0
\qquad
(1\le i,j\le4).
$$
Since these are polynomial identities, they hold for every $x$. For
$i\ne j$, the coprimality of $x_i$ and $x_j$ implies that $x_i$ divides
$\mathcal C_{W',i}$. Writing
$\mathcal C_{W',i}=x_iQ_i$, the same identities give $Q_i=Q_j$ for all
$i,j$. Thus there is a quadratic form $Q_W$ such that
\begin{equation}\label{eq:cofactor-reconstruction}
\mathcal C_{W'}(x)=Q_W(x)x.
\end{equation}

Changing the basis of $W'$ multiplies $Q_W$ by a nonzero scalar. The
coefficients of $Q_W$ depend polynomially on a basis of $W'$. Choosing
this basis by orthogonally projecting a fixed basis of $W_0$ therefore
gives
$$
\operatorname{dist}([Q_{\Delta,\mathcal B}],[Q_W])
\ll
\operatorname{dist}(W,\mathcal M_c).
$$

By Minkowski's second theorem and the fixed positive lower bound for the
norm of a nonzero vector of $\wedge^2\mathbb Z^4$, the lattice
$W'\cap\wedge^2\mathbb Z^4$ has a basis
$\omega_1,\omega_2,\omega_3$ satisfying
$$
\|\omega_i\|
\ll_{\mathcal B}
\|\mathsf w_{\wedge^2\Delta,W}\|.
$$
For this basis, the coefficients of $L_{W'}$ are integral and have size
$
O\bigl(\|\mathsf w_{\wedge^2\Delta,W}\|\bigr).
$
The cofactor formula therefore gives an integral multiple of $Q_W$ whose
coefficients have size
$
O\bigl(\|\mathsf w_{\wedge^2\Delta,W}\|^3\bigr),
$
which proves the height bound.

Finally,
$(\wedge^2g_{\mathcal B}^{-1})(u\wedge v)\in W'$. If
$
x\in
\operatorname{span}
\{g_{\mathcal B}^{-1}u,g_{\mathcal B}^{-1}v\},
$
then one linear combination of the rows of $L_{W'}(x)$ vanishes
identically as a linear form in $y$. Thus
$\operatorname{rank}L_{W'}(x)\le2$, and
\eqref{eq:cofactor-reconstruction} gives $Q_W(x)=0$. Hence $Q_W$
vanishes on a rational two-plane. For $W$ sufficiently close to $\mathcal M_c$, the form $Q_W$ is
nondegenerate. Since it vanishes on a rational two-plane, it is split over
$\mathbb Q$. The proof for $\mathcal M_r$ is identical.
\end{proof}

For $s\ge0$ and $0<\varepsilon\le1$, define
$$
\Xi(s,\varepsilon)
:=
\left\{
0\ne w\in E:
\|w\|\le e^{4(s+1)},\
\min\{\|w-\pi_cw\|,\|w-\pi_rw\|\}\le\varepsilon
\right\}.
$$
Let $\mathcal K_{\eta,M}(s,\varepsilon)$ be the set of pairs $(h,\Delta)$ for which there is a $\Delta$-rational two-plane $V$, with $w=\mathsf w_{\Delta,V}\notin\mathcal H_{\eta,M}$, such that $hw\in\Xi(s,\varepsilon)$.

We recall the quantitative-nondivergence terminology. A continuous
function $f$ on a ball $B\subset\mathbb R^d$ is called
\emph{$(C,\tau)$-good} if, for every ball $B'\subset B$ satisfying
$\sup_{B'}|f|>0$ and every $u>0$,
$$
\Leb\{x\in B':|f(x)|<u\}
\le
C
\left(
\frac{u}{\sup_{B'}|f|}
\right)^\tau
\Leb(B').
$$
A map $g$ is called $(C,\tau)$-good if
$n\mapsto\|(\wedge^rg(n))v\|$ is $(C,\tau)$-good for every exterior
degree $r$ and every $v\in\wedge^rE$.

\begin{lem}[Quantitative nondivergence for selected vectors]
\label{lem:d2-prescribed-qnd}
For every integer $d_0\ge1$, there exist constants $C,\tau>0$, depending
only on $E$, $d_0$, and $B_N(1)$, such that the following holds. Let
$\mathcal L<E$ be a lattice, let $\Upsilon\subset\mathcal L$ be a
collection of primitive vectors, and let
$$
g:B_N(3^6)\to \GL(E)
$$
be a polynomial map whose coordinate functions have degree at most $d_0$
and which is $(C,\tau)$-good. If
$0<\varepsilon\le\rho\le1$, then either
$$
\Leb\left\{
n\in B_N(1):
\min_{w\in\Upsilon}\|g(n)w\|\le\varepsilon
\right\}
\ll
(\varepsilon/\rho)^\tau,
$$
or there are linearly independent vectors
$w_1,\ldots,w_m\in\Upsilon$, with $1\le m\le6$, such that
$$
\sup_{n\in B_N(1)}
\|(\wedge^mg(n))(w_1\wedge\cdots\wedge w_m)\|
<
\rho^m.
$$
\end{lem}

\begin{proof}
This follows from the quantitative nondivergence theorem
\cite[Theorem~2.1]{kleinbock:2008}, with the argument restricted to the
selected primitive vectors in $\Upsilon$. See
\cite[Proposition~7.3]{KimOh_det} for details.
\end{proof}

We shall also use the following elementary expansion property of
nontrivial irreducible $H$-modules.

\begin{lem}[Instability]
\label{lem:d2-instability}
Let $U$ be a nontrivial irreducible $H$-subrepresentation of an exterior
power of $E$. There exist constants $c,\kappa>0$ such that
$$
\sup_{n\in B_N(1)}\|a_tnu\|
\ge
ce^{\kappa t}\|u\|
\qquad
(u\in U,\ t\ge0).
$$
\end{lem}

\begin{proof}
Write
$
U
\simeq
\operatorname{Sym}^p(\mathbb R^2)
\otimes
\operatorname{Sym}^q(\mathbb R^2)$ where $p+q>0$,
and realize the two factors as spaces of homogeneous polynomials. The
coefficient of $Y^p\otimes Y^q$ in $n_{\xi_1,\xi_2}u$ is a polynomial in
$(\xi_1,\xi_2)$ whose coefficients are nonzero scalar multiples of the
coordinates of $u$. In particular, this polynomial is not identically zero
when $u\ne0$.

The resulting coefficient map from $U$ into a finite-dimensional
polynomial space is injective. Compactness of the unit sphere in $U$
therefore gives a uniform lower bound for the supremum of this coefficient
on $[-1,1]^2$. Since $a_t$ multiplies
$Y^p\otimes Y^q$ by $e^{(p+q)t}$, the result follows.
\end{proof}

We can now prove the basic avoidance estimate.

\begin{prop}[Avoidance of near-exceptional directions]
\label{prop:avoidance0}
There exists $B_0=B_0(\Delta,\eta,M)>1$ such that, if
$
0<\varepsilon<e^{-B_0Ms}$ and $
t\ge B_0\log(1/\varepsilon),
$
then
\begin{equation}\label{eq:d2-avoidance0}
\Leb\left\{
n\in B_N(1):
(a_tn,\Delta)\in\mathcal K_{\eta,M}(s,\varepsilon)
\right\}
\ll
\varepsilon^{1/(B_0M)}.
\end{equation}
\end{prop}

\begin{proof}
We treat the condition
$
\|a_tnw-\pi_c(a_tnw)\|\le\varepsilon;
$
the argument for $\mathcal M_r$ is identical. Let $\Upsilon$ be the
collection of primitive Pl\"ucker vectors of $\Delta$-rational two-planes
that do not belong to $\mathcal H_{\eta,M}$. Define $D\in\GL(E)$ by
$$
D|_{\mathcal M_c}
=
\frac{\varepsilon}{2e^{4(s+1)}}\operatorname{id},
\qquad
D|_{\mathcal M_r}
=
\frac12\operatorname{id}.
$$
If $a_tnw\in\Xi(s,\varepsilon)$ and
$\|a_tnw-\pi_c(a_tnw)\|\le\varepsilon$, then
$\|Da_tnw\|\le\varepsilon$.

For every exterior degree, the coordinates of
$\wedge^r(Da_tn)$ are polynomials of degree bounded independently of
$s,t,\varepsilon$. The vector-valued Remez inequality supplies the
goodness constants required in
\Cref{lem:d2-prescribed-qnd}.

Choose
$$
0<\delta<
\min\left\{
\frac1{100},\frac1{40M_0}
\right\},
\qquad
\rho=\varepsilon^{1-\delta}.
$$
By \Cref{lem:d2-prescribed-qnd}, either the measure estimate in
\eqref{eq:d2-avoidance0} follows, after enlarging $B_0$, or there are
linearly independent vectors $w_1,\ldots,w_m\in\Upsilon$ such that, with
$
\mathbf w=w_1\wedge\cdots\wedge w_m,
$
one has
\begin{equation}\label{eq:d2-qnd-obstruction}
\sup_{n\in B_N(1)}
\|(\wedge^m(Da_tn))\mathbf w\|
<
\rho^m.
\end{equation}

Decompose
$$
\mathbf w
=
\sum_{j=0}^m\mathbf w_j,
\qquad
\mathbf w_j
\in
\wedge^j\mathcal M_c
\otimes
\wedge^{m-j}\mathcal M_r.
$$
In each of these finitely many $H$-modules, write
$\mathbf w_j$ as the sum of its $H$-fixed part
$\mathbf w_j^{\operatorname{fix}}$ and its projections
$\mathbf w_j^{\operatorname{nt}}$ onto the nontrivial irreducible summands.
Let $\kappa>0$ be the minimum of the constants supplied by
\Cref{lem:d2-instability} over these nontrivial summands.

Projecting \eqref{eq:d2-qnd-obstruction} and undoing the action of $D$
gives
$$
\|\mathbf w_j^{\operatorname{fix}}\|
\ll
e^{4j(s+1)}\varepsilon^{m(1-\delta)-j},
\quad\text{and}\quad
\|\mathbf w_j^{\operatorname{nt}}\|
\ll
e^{4j(s+1)}\varepsilon^{m(1-\delta)-j}e^{-\kappa t}.
$$
Since $\varepsilon<e^{-B_0Ms}$ and
$t\ge B_0\log(1/\varepsilon)$,
$$
e^{4j(s+1)}
\ll
\varepsilon^{-4j/(B_0M)},
\qquad
e^{-\kappa t}
\le
\varepsilon^{\kappa B_0}.
$$
After enlarging $B_0$, we may therefore arrange that
$$
m(1-\delta)-j-\tfrac{4j}{B_0M}
>
\tfrac12 \quad (0\le j <m) ;$$
$$ m(1-\delta)-j+\kappa B_0-\tfrac{4j}{B_0M}
>
\tfrac12
\quad 
(0\le j\le m).
$$
It follows that all terms with $j<m$, as well as all nontrivial terms with
$j=m$, have norm $O(\varepsilon^{1/2})$. For the fixed part with $j=m$,
enlarging $B_0$ once more gives
$$
\|\mathbf w_m^{\operatorname{fix}}\|
\ll
e^{4m(s+1)}\varepsilon^{-m\delta}
\ll
\varepsilon^{-2m\delta}.
$$
Consequently,
\begin{equation}\label{eq:d2-wedge-ruling-close}
\|\mathbf w-(\wedge^m\pi_c)\mathbf w\|
\ll
\varepsilon^{1/2},
\qquad
\|\mathbf w\|
\ll
\varepsilon^{-2m\delta}.
\end{equation}

The modules $\mathcal M_c$ and $\wedge^2\mathcal M_c$ have no nonzero
$H$-fixed vectors, whereas $\wedge^3\mathcal M_c$ is one-dimensional and
fixed. Since every nonzero vector in each of the finitely many lattices
$\wedge^m(\wedge^2\Delta)$ has norm bounded below,
\eqref{eq:d2-wedge-ruling-close} is impossible when $m>3$, because
$\wedge^m\pi_c=0$. It is also impossible when $m=1$ or $m=2$, because in
those cases the pure $\mathcal M_c$-component has no fixed part and is
therefore also $O(\varepsilon^{1/2})$. Hence $m=3$.

Let
$
W=\operatorname{span}\{w_1,w_2,w_3\}.
$
Since $\mathbf w$ is a nonzero integral multiple of
$\mathsf w_{\wedge^2\Delta,W}$,
\eqref{eq:d2-wedge-ruling-close} gives
$$
\|\mathsf w_{\wedge^2\Delta,W}\|
\ll
\varepsilon^{-6\delta}.
$$
Moreover, $(\wedge^3\pi_c)\mathbf w\ne0$ for sufficiently small
$\varepsilon$, and normalization in
\eqref{eq:d2-wedge-ruling-close} gives
$$
\operatorname{dist}(W,\mathcal M_c)
\ll_\Delta
\varepsilon^{1/2}.
$$

The vector $w_1$ is the wedge of a $\mathbb Z$-basis of a rank-two
sublattice of $\Delta$. Hence, for sufficiently small $\varepsilon$,
\Cref{lem:rational-ruling-reconstruction-d2} gives a rational split form
$Q_W$ such that
$$
\operatorname{ht}(Q_W)
\ll
\varepsilon^{-18\delta},
\qquad
\operatorname{dist}([Q_{\Delta,\mathcal B}],[Q_W])
\ll
\varepsilon^{1/2}.
$$
On the other hand, the split-Diophantine condition gives
$$
\operatorname{dist}([Q_{\Delta,\mathcal B}],[Q_W])
\ge
\eta_0\operatorname{ht}(Q_W)^{-M_0}
\gg
\varepsilon^{18M_0\delta}.
$$
These two estimates are incompatible because
$18M_0\delta<1/2$.

Choose $\varepsilon_0>0$ so that this contradiction holds whenever
$0<\varepsilon<\varepsilon_0$. For
$\varepsilon_0\le\varepsilon<1$, the desired estimate follows after
increasing the implicit constant, since $\Leb(B_N(1))<\infty$. Thus the
second alternative in \Cref{lem:d2-prescribed-qnd} cannot occur. The first
alternative gives a bound of order $\varepsilon^{\tau\delta}$; enlarging
$B_0$ proves \eqref{eq:d2-avoidance0}.
\end{proof}

We now incorporate the ordinary height into the avoidance estimate. This is
the form needed in the stopping-time argument of the next section.

\begin{prop}[Avoidance estimate]
\label{prop:avoidance'}
There exists $B=B(\Delta,\eta,M)>1$ such that for all
$M'\ge BM^2, s\ge0, t\ge BM's$,
one has
\begin{equation}\label{eq:d2-avoidance}
\int_{B_N(1)}
\alpha(a_tn\Delta)^2
\mathbf 1_{\mathcal E_{s,\eta,M,M'}}(a_tn,\Delta)\,dn
\ll
e^{-M's/(BM)}+e^{-t/(BM)}.
\end{equation}
\end{prop}

\begin{proof}
Let $B_0$ be the constant in \Cref{prop:avoidance0}. Fix $B_1>8B_0^2$. After enlarging $B$, there is $C_1>1$ such that
\begin{equation}\label{eq:d2-exceptional-alpha-cap}
(a_tn,\Delta)\in\mathcal E_{s,\eta,M,M'}\quad\Rightarrow\quad 
\alpha(a_tn\Delta)\le C_1e^{t/(B_1M)}.
\end{equation}
In the first case of \eqref{eq:d2-exceptional-scale}, this follows from $\alpha(a_tn\Delta)\le e^{5(s+1)}$ and $t\ge BM's$. In the second case, a vector $w\notin\mathcal H_{\eta,M}$ witnessing exceptionality satisfies
$$
\|w\|\ll e^{4(t+1)} e^{4(s+1)},
\qquad
\min\{\|w-\pi_cw\|,\|w-\pi_rw\|\}\ll e^{4(t+1)}\alpha(a_tn\Delta)^{-M'}.
$$
If \eqref{eq:d2-exceptional-alpha-cap} failed, then, after increasing $B$ and $C_1$,
$$
 e^{4(t+1)} \alpha(a_tn\Delta)^{-M'}\le\eta(e^{4(t+1)} e^{4(s+1)})^{-M}\le\eta\|w\|^{-M},
$$
because $M'\ge BM^2$ and $t\ge BM's$. This contradicts $w\notin\mathcal H_{\eta,M}$.

Put $\mathsf h_j=e^{10(s+1)}e^{10j}$, and let $J$ be the least integer such that $\mathsf h_J\ge C_1e^{t/(B_1M)}$. For an exceptional point, choose the least $j$ such that $\alpha(a_tn\Delta)\le\mathsf h_j$. If $j=0$, enlarge the exceptional scale to $e^{-M's}$. If $j\ge1$, then $\alpha(a_tn\Delta)>\mathsf h_{j-1}>e^{5(s+1)}$, so the second line of \eqref{eq:d2-exceptional-scale} applies and $\alpha(a_tn\Delta)^{-M'}<\mathsf h_{j-1}^{-M'}$. Hence every exceptional point belongs to one of the sets
$$
E_0:=\{n\in B_N(1):\alpha(a_tn\Delta)\le\mathsf h_0,\ (a_tn,\Delta)\in\mathcal K_{\eta,M}(s,e^{-M's})\},
$$
$$
E_j:=\left\{n\in B_N(1):
\alpha(a_tn\Delta)\le\mathsf h_j ,\;\;
(a_tn,\Delta)\in\mathcal K_{\eta,M}(s,\mathsf h_{j-1}^{-M'}) \right\},
\qquad1\le j\le J.
$$
Consequently,
\begin{equation}\label{eq:d2-level-sum}
\int_{B_N(1)}\alpha(a_tn\Delta)^2\mathbf1_{\mathcal E_{s,\eta,M,M'}}(a_tn,\Delta)\,dn
\ll\mathsf h_0^2\Leb(E_0)+\sum_{j=1}^J\mathsf h_j^2\Leb(E_j).
\end{equation}

If $s=0$, the first term is bounded by a constant. If $s>0$, the hypotheses on $M'$ and $t$ imply that \Cref{prop:avoidance0} applies with $\varepsilon=e^{-M's}$, and hence
$$
\Leb(E_0)\ll e^{-M's/(B_0M)}.
$$
Call $j\ge1$ low if $B_0M'\log\mathsf h_{j-1}\le t/2$. For such $j$, one also has $\mathsf h_{j-1}^{-M'}<e^{-B_0Ms}$, so \Cref{prop:avoidance0} gives $\Leb(E_j)\ll\mathsf h_{j-1}^{-M'/(B_0M)}$, hence
$$
\mathsf h_j^2\Leb(E_j)\ll\mathsf h_{j-1}^{2-M'/(B_0M)}.
$$
Since $\mathsf h_j/\mathsf h_{j-1}=e^{10}$ and $M'\ge BM^2$, the contribution of $E_0$ and the low indices to \eqref{eq:d2-level-sum} is $\ll e^{-M's/(BM)}$, after enlarging $B$.

If $t$ is bounded, \eqref{eq:d2-exceptional-alpha-cap} proves the required estimate after increasing the implicit constant. We may therefore assume that $t$ is sufficiently large. For the non-low indices, put $\varepsilon_*=e^{-t/(2B_0)}$. Then $\mathsf h_{j-1}^{-M'}<\varepsilon_*$. The assumptions on $M'$, $s$, and $t$ also give $\varepsilon_*<e^{-B_0Ms}$, while $t\ge B_0\log(1/\varepsilon_*)$. Thus \Cref{prop:avoidance0} gives
$$
\Leb(E_j)\ll e^{-t/(2B_0^2M)}.
$$
By the minimality of $J$, we have $\mathsf h_J^2\ll e^{20(s+1)}e^{2t/(B_1M)}$ and $J+1\ll1+s+t$. Therefore the contribution of the non-low indices is at most
$$
e^{20(s+1)}(J+1)\exp\left[-\left(\frac1{2B_0^2}-\frac2{B_1}\right)\frac tM\right].
$$
Put $c=1/(2B_0^2)-2/B_1>0$. After enlarging $B$, the assumptions $M'\ge BM^2$ and $t\ge BM's$ imply $20(s+1)+\log(J+1)\le\frac{ct}{2M}$ whenever the non-low range is nonempty. Hence its contribution is $\ll e^{-ct/(2M)}\ll e^{-t/(BM)}$. This proves \eqref{eq:d2-avoidance}.
\end{proof}
\section{Uniform moment bounds for the modified height}
\label{s:iterations}

We now combine the global contraction estimate with the avoidance estimate
through a stopping-time iteration, following
\cite[Section~6]{Kim}. We first prove a uniform moment bound for averages
over the expanding horospherical subgroup $N$. A local coordinate argument
then transfers this bound to the $K$-averages needed in the counting
problem.

Choose $R_0>1$ large enough that the unipotent parameters in the finitely
many local coordinate charts used in
\cite[Subsection~7.2]{Kim} belong to $B_N(R_0)$. For every fixed $R>0$,
the estimates in \Cref{prop:globalcontraction,prop:avoidance'} and the
ordinary-height estimate below remain valid with $B_N(1)$ replaced by
$B_N(R)$, after changing the constants. We use the contraction and
ordinary-height estimates on $B_N(2R_0)$ and the avoidance estimate on
$B_N(3R_0)$. Let $\nu$ be the probability measure obtained by restricting
$dn$ to $B_N(2R_0)$ and normalizing.

The following comparison allows us to replace a single long
horospherical average by a product of shorter averages.

\begin{lem}[Random walk comparison]
\label{lem:randomwalkgeneral}
There is $T_0>0$ such that, if $s_1,\ldots,s_m\ge T_0$, then, for every
nonnegative measurable function $F:H\times X\to[0,\infty)$ and every
$\Delta\in X$,
\begin{multline*}
\int_{B_N(R_0)}
F(a_{s_1+\cdots+s_m}n;\Delta)\,dn\\
\ll
\int_{B_N(2R_0)^m}
F(a_{s_m}n_m\cdots a_{s_1}n_1;\Delta)\,
d\nu(n_1)\cdots d\nu(n_m),
\end{multline*}
and
\begin{multline*}
\int_{B_N(2R_0)^m}
F(a_{s_m}n_m\cdots a_{s_1}n_1;\Delta)\,
d\nu(n_1)\cdots d\nu(n_m)\\
\ll
\int_{B_N(3R_0)}
F(a_{s_1+\cdots+s_m}n;\Delta)\,dn.
\end{multline*}
\end{lem}

\begin{proof}
This is the product-group version of \cite[Lemma~6.1]{Kim}. Writing
$n_j=n_{\xi_j}$ with $\xi_j\in\mathbb R^2$, we have
$$
a_{s_m}n_m\cdots a_{s_1}n_1
=
a_{s_1+\cdots+s_m}
n_{\xi_1+e^{-2s_1}\xi_2+\cdots
+e^{-2(s_1+\cdots+s_{m-1})}\xi_m}.
$$
For sufficiently large $T_0$, the resulting convolution density is
uniformly bounded, is supported in $B_N(3R_0)$, and is bounded below on
$B_N(R_0)$. 
\end{proof}

We shall choose the lengths of the successive steps using the following
elementary decomposition (see \cite[Lemma~6.2]{Kim}).

\begin{lem}[Step decomposition]
\label{lem:ndecomposition}
Let $D>1$, $0<\delta<(1+D)^{-1}$, and $T>0$. For every
$t\ge T/\delta$, there are $I\ge2$ and $s_1,\ldots,s_I\ge T$ such that
$$
\begin{gathered}
t=s_1+\cdots+s_I,\qquad s_1=Ds_2,\\
s_i=(1+\delta)s_{i+1}\quad(2\le i<I),
\qquad
T\le s_I\le2T.
\end{gathered}
$$
\end{lem}

We also use a bounded-expansion estimate for the ordinary height. There are
constants $C>0$ and $\gamma_0>0$, depending on $R_0$, such that, for
$0<\gamma\le\gamma_0$ and $s\geq 1$,
\begin{equation}\label{eq:ordinary-height-small-moment}
\int_{B_N(2R_0)}
\alpha(a_snh\Delta)^{1+\gamma}\,d\nu(n)
\le
e^{C\gamma s}\alpha(h\Delta)^{1+\gamma}
+
e^{Cs}.
\end{equation}
This follows from the fixed-subspace estimates and the Pl\"ucker covolume
inequality; see \cite[Section~5]{eskin-margulis-mozes:1998}.

\begin{prop}[Uniform $N$-moment bound]
\label{prop:maintheoremforu}
Suppose that $\Lambda\in X$ is Diophantine and not
determinant-rational. Choose $0<\eta<1$ and $M>1$ such that \eqref{fs}
holds. Then there is $\theta'>0$ such that
\begin{equation}\label{eq:d2-N-average-main}
\sup_{t\ge0}
\int_{B_N(R_0)}
\widehat\alpha_{\eta,M}(a_tn;\Lambda)^{1+\theta'}\,dn
<\infty.
\end{equation}
\end{prop}

\begin{proof}
Let $B$ be the constant in the fixed-ball form of
\Cref{prop:avoidance'}, and let $C,\gamma_0$ be as in
\eqref{eq:ordinary-height-small-moment}. Put
$$
M'=100B(C+1)M^2.
$$
Choose $0<\theta<10^{-4}$ such that
$M'\sqrt\theta\le1/20$, and let $C_\theta$ be the constant in the
fixed-ball form of \Cref{prop:globalcontraction}. Next, choose
$$
D\ge4BM',
\qquad
\delta=\frac{\theta}{100D(C_\theta+1)}.
$$
Finally, choose
$0<\theta'\le\min\{\theta,\gamma_0\}$ sufficiently small that
\begin{equation}\label{eq:theta-prime-iteration}
2C\theta'\delta^{-1}+C
<
\frac{M'}{2BM}.
\end{equation}
Such a choice is possible by the definition of $M'$. Choose $T\ge\max\{T_0,e\}$ sufficiently large that
$$
C_\theta e^{-\theta s/2}
\le
e^{-\theta s/4}
\qquad
(s\ge T),
$$
and so that \Cref{lem:d2-quasinull-at-identity} applies for $s\ge T$.
Since $M'\ge5M$, we have
$$
\mathcal E_{s,\eta,M,M'}
\subset
\mathcal E_{s,\eta,M,5M}.
$$
It follows that
$(e,\Lambda)\notin\mathcal E_{s,\eta,M,M'}$ for every $s\ge T$.

For $0\le t<T/\delta$, the required estimate follows from compactness and
the log-Lipschitz estimates. Suppose, therefore, that $t\ge T/\delta$, and
choose
$$
t=s_1+\cdots+s_I
$$
as in \Cref{lem:ndecomposition}.
For $n_j\in B_N(2R_0)$, put
$$
h_m
=
a_{s_m}n_m\cdots a_{s_1}n_1
\qquad
(1\le m\le I),
$$
and define
$$
\mathsf Z_t
:=
\int_{B_N(2R_0)^I}
\widehat\alpha_{\eta,M}(h_I;\Lambda)^{1+\theta'}\,
d\nu(n_1)\cdots d\nu(n_I).
$$
By \Cref{lem:randomwalkgeneral}, it is enough to prove that
$\sup_t\mathsf Z_t<\infty$.

For $1\le m<I$, let
$$
\Theta_m
:=
\left\{
(n_1,\ldots,n_I)\in B_N(2R_0)^I:
(h_m,\Lambda)\in
\mathcal E_{s_{m+1},\eta,M,M'}
\right\},
$$
and put $\Theta=\bigcup_{m=1}^{I-1}\Theta_m$. Thus
$$
\mathsf Z_t
\le
\mathsf Y_t+\sum_{m=1}^{I-1}\mathsf E_{t,m},
$$
where $\mathsf Y_t$ is the integral over
$B_N(2R_0)^I\smallsetminus\Theta$, and $\mathsf E_{t,m}$ is the integral over
$\Theta_m$.

We first estimate the exceptional contributions. By Minkowski's theorem,
$\alpha$ is bounded below by a positive constant on $X$. Hence
$$
\widehat\alpha_{\eta,M}(h;\Lambda)
\le
\max\{1,\alpha(h\Lambda)\}
\ll
\alpha(h\Lambda).
$$
Iterating \eqref{eq:ordinary-height-small-moment}, as in
\cite[proof of Theorem~1.9, especially (6.6)]{Kim}, gives, for every fixed
prefix $(n_1,\ldots,n_m)$,
\begin{align*}
&\int_{B_N(2R_0)^{I-m}}
\widehat\alpha_{\eta,M}(h_I;\Lambda)^{1+\theta'}\,
d\nu(n_{m+1})\cdots d\nu(n_I)\\
&\qquad\ll
\delta^{-1}\log(e+s_{m+1})\,
e^{(2C\theta'\delta^{-1}+C)s_{m+1}}
\bigl(\alpha(h_m\Lambda)^2+1\bigr).
\end{align*}
By \eqref{eq:theta-prime-iteration}, the exponential factor is at most
$
e^{M's_{m+1}/(2BM)}.
$

Put $S_m=s_1+\cdots+s_m$. Since
$$
S_m\ge s_1=Ds_2\ge BM's_{m+1},
$$
\Cref{lem:randomwalkgeneral} and the fixed-ball form of
\Cref{prop:avoidance'} give
\begin{multline*}
\int_{B_N(2R_0)^m}
\bigl(\alpha(h_m\Lambda)^2+1\bigr)
\mathbf 1_{\mathcal E_{s_{m+1},\eta,M,M'}}(h_m,\Lambda)\\
d\nu(n_1)\cdots d\nu(n_m)
\ll
e^{-M's_{m+1}/(BM)}.
\end{multline*}
Consequently, after increasing $T$,
$$
\mathsf E_{t,m}
\ll
\delta^{-1}\log(e+s_{m+1})
e^{-M's_{m+1}/(2BM)}
\ll
e^{-s_{m+1}}.
$$
Since
$s_{m+1}\ge(1+\delta)^{I-m-1}T$, we obtain
\begin{equation}\label{eq:iteration-exceptional-sum}
\sum_{m=1}^{I-1}\mathsf E_{t,m}
\ll
\sum_{j=0}^{\infty}e^{-(1+\delta)^jT}
<\infty.
\end{equation}

We next estimate the contribution from the complement of $\Theta$. The
choice of $T$ ensures that $(e,\Lambda)$ is outside the exceptional set
needed for the first step, while the definition of $\Theta$ ensures that
each subsequent prefix is outside the exceptional set needed for the next
step. Since $\theta'\le\theta$, the comparison
\eqref{eq:auxiliary-comparisons} and successive applications of the
fixed-ball form of \Cref{prop:globalcontraction} give
\begin{equation}\label{eq:iteration-good-bound}
\mathsf Y_t
\ll
e^{-\theta t/4}
\widetilde\alpha_{\eta,M,\theta}(e;\Lambda)
+
\sum_{m=1}^I
\exp\left(
C_\theta s_m
-
\frac{\theta}{4}(s_{m+1}+\cdots+s_I)
\right),
\end{equation}
where an empty sum in the exponent is understood to be zero.

Let $I_0$ be the largest index such that $s_{I_0}\ge2s_I$. For
$2\le m\le I_0$, the geometric relation among the $s_j$ gives
$$
s_{m+1}+\cdots+s_I
\ge
\frac{s_m}{2\delta}.
$$
If $I_0\ge2$, then similarly
$$
s_2+\cdots+s_I
\ge
\frac{s_1}{2D\delta}.
$$
By the choice of $\delta$, the terms in the sum in
\eqref{eq:iteration-good-bound} with $m\le I_0$ are
$O(e^{-s_m})$. If $I_0=1$, then $s_2<2s_I$, and hence
$$
s_1=Ds_2<2Ds_I\le4DT;
$$
thus the $m=1$ term in the error sum is bounded independently of $t$.

Finally,
$$
I-I_0\ll\delta^{-1},
\qquad
s_m<2s_I\le4T
\quad
(m>I_0).
$$
The remaining terms in the error sum are therefore bounded independently
of $t$. It follows from \eqref{eq:iteration-good-bound} that
$\sup_t\mathsf Y_t<\infty$. Together with
\eqref{eq:iteration-exceptional-sum}, this proves that
$\sup_t\mathsf Z_t<\infty$. The first inequality in
\Cref{lem:randomwalkgeneral} now yields
\eqref{eq:d2-N-average-main}.
\end{proof}

It remains to transfer the uniform bound from the horospherical averages
to the $K$-averages.

\begin{lem}[From $N$-averages to $K$-averages]
\label{lem:K-from-N-d2}
There exist $R_0>1$ and elements
$k_1,\ldots,k_q\in K$ such that the following holds.
Let $F:H\to[1,\infty)$ be measurable and suppose that, for every compact
set $\mathcal C\subset H$, there is $A_{\mathcal C}>0$ such that
\[
F(gh)\le A_{\mathcal C}F(h)
\qquad(g\in\mathcal C,\ h\in H).
\]
If
\[
\sup_{t\ge0}\int_{B_N(R_0)}F(a_tnk_j)\,dn<\infty
\qquad(1\le j\le q),
\]
then
\[
\sup_{t\ge0}\int_KF(a_tk)\,dk<\infty.
\]
\end{lem}

\begin{proof}
This is the product-group version of the local
parabolic-unipotent coordinate argument at the beginning of
\cite[Subsection~7.2]{Kim}. In each of finitely many coordinate charts,
one may write
$$
k=p_\gamma(k)n_{r_\gamma(k)}\gamma,
$$
where $r_\gamma(k)\in B_N(R_0)$, the coordinate Jacobians are uniformly
bounded, and
$a_tp_\gamma(k)a_{-t}$ ranges over a fixed compact subset of $H$.
The compact-set hypothesis on $F$ therefore reduces each chart integral to
an $N$-average over $B_N(R_0)$.
\end{proof}

\begin{theorem}[Uniform moment bound]
\label{thm:moment-d2}
Let $\Lambda\in X$ be Diophantine and not determinant-rational. Choose
$0<\eta<1$ and $M>1$ such that \eqref{fs} holds. Then there exists
$\theta'>0$ such that
$$
\sup_{t\ge0}
\int_K
\widehat\alpha_{\eta,M}(a_tk;\Lambda)^{1+\theta'}\,dk
<\infty.
$$
\end{theorem}

\begin{proof}
Apply \Cref{lem:K-from-N-d2} to
$
F(h)
=
\widehat\alpha_{\eta,M}(h;\Lambda)^{1+\theta'},
$
using \Cref{prop:maintheoremforu}. The compact-set hypothesis follows from
the log-Lipschitz estimates on the finitely many exterior powers.
\end{proof}

\section{Quasi-null contribution and modified Siegel transforms}
\label{sec:modified-siegel-d2}

The moment bound in \Cref{thm:moment-d2} controls a modified height that
omits all quasi-null rational two-planes. For the counting problem,
however, we remove only the exact rational isotropic planes, which produce
the singular contribution treated in
Section~\ref{sec:singular-d2}. We show that the remaining quasi-null but
non-isotropic planes contribute negligibly in the limit. We then use the
uniform moment bound to extend Shah's equidistribution theorem
\Cref{shah} from bounded continuous functions to the modified Siegel
transform.

Throughout the section, fix a Diophantine lattice
$\Lambda\in X$ that is not determinant-rational. Choose
$0<\eta_0<1$ and $M>1$ such that \eqref{fs} holds, and fix
$$
0<\eta\le2^{-M}\eta_0.
$$
All quasi-null planes below are defined using this smaller parameter
$\eta$. Since
$$
\mathcal Q_\Lambda(R,\eta R^{-M})
\subset
\mathcal Q_\Lambda(R,\eta_0R^{-M}),
$$
the finite-scale condition \eqref{fs}, and hence
\Cref{thm:moment-d2}, remain valid with the parameters $\eta$ and $M$.

For a $\Lambda$-rational two-plane $V$, put
$w_V=\mathsf w_{\Lambda,V}$. Define
$$
\mathcal X_{\eta,M}(\Lambda)
=
\left\{
V:
V\text{ is a $\Lambda$-rational two-plane and }
w_V\in\mathcal H_{\eta,M}
\right\},
$$
and
$$
\widehat{\mathcal X}_{\eta,M}(\Lambda)
=
\mathcal X_{\eta,M}(\Lambda)\smallsetminus\mathcal I_\Lambda.
$$
Thus $\widehat{\mathcal X}_{\eta,M}(\Lambda)$ is the collection of
quasi-null rational two-planes that are not exactly isotropic, while
$\mathcal I_\Lambda$ is the finite collection of exact rational
column- and row-isotropic planes.

We begin with the critical first-moment estimate.

\begin{lem}[Critical first-moment estimate]
\label{lem:critical-K-average-d2}
There is a constant $C\ge1$ such that, for every
$0\ne w\in\wedge^2\M_2(\mathbb R)$ and every $t\ge0$,
\begin{equation}\label{eq:critical-K-average-d2}
\int_K\|a_tkw\|^{-1}\,dk
\le
C\|w\|^{-1}.
\end{equation}
Moreover, if $w\notin\mathcal M_c\cup\mathcal M_r$, then
\begin{equation}\label{eq:critical-K-average-decay-d2}
\lim_{t\to\infty}
\int_K\|a_tkw\|^{-1}\,dk
=
0.
\end{equation}
\end{lem}

\begin{proof}
By \Cref{prop:product-local-d2}\textnormal{(ii)} with $\beta=1$, for
every fixed $R>0$,
$$
\int_{B_N(R)}
\|a_tnw\|^{-1}\,dn
\ll_R
\|w\|^{-1}.
$$
On each of the finitely many local charts used at the beginning of
\cite[Subsection~7.2]{Kim}, one has
$$
k=p_\gamma(k)n_{r_\gamma(k)}\gamma,
\qquad
a_tk
=
(a_tp_\gamma(k)a_{-t})
a_tn_{r_\gamma(k)}\gamma,
$$
where $a_tp_\gamma(k)a_{-t}$ ranges over a fixed compact subset of $H$,
the parameter $r_\gamma(k)$ lies in a fixed ball in $N$, and the coordinate
Jacobians are uniformly bounded. Since multiplication by elements of a
fixed compact set distorts the norm by only a bounded factor, the
fixed-ball estimate gives \eqref{eq:critical-K-average-d2}.

Now suppose that
$w\notin\mathcal M_c\cup\mathcal M_r$. For every sufficiently small
$\theta>0$,
$$
\|v\|^{-1}
\ll_\theta
\phi_\theta(v)
\qquad
\left(
v\notin\mathcal M_c\cup\mathcal M_r
\right).
$$
The fixed-ball form of \Cref{lem:weighted-local-d2} and H\"older's
inequality therefore give
$$
\begin{aligned}
\int_{B_N(R)}
\|a_tnw\|^{-1}\,dn
&\ll_{\theta,R}
\left(
\int_{B_N(R)}
\phi_\theta(a_tnw)^{1+\theta}\,dn
\right)^{1/(1+\theta)}\\
&\ll_{\theta,R}
e^{-\theta t/(1+\theta)}\phi_\theta(w).
\end{aligned}
$$
The summands $\mathcal M_c$ and $\mathcal M_r$ are $H$-invariant, so
$\gamma w\notin\mathcal M_c\cup\mathcal M_r$ in each chart. Applying the
same chart decomposition proves
\eqref{eq:critical-K-average-decay-d2}.
\end{proof}

The finite-scale bound now allows us to sum the preceding estimate over
all non-isotropic quasi-null planes.

\begin{prop}[Vanishing of non-isotropic quasi-null planes]
\label{prop:nonisotropic-quasinull-vanishing-d2}
Assume that $\Lambda$ satisfies \eqref{fs} for $\eta_0$ and $M$, and let
$0<\eta\le2^{-M}\eta_0$. Then
\begin{equation}\label{eq:nonisotropic-quasinull-vanishing-d2}
\lim_{t\to\infty}
\int_K
\sum_{V\in\widehat{\mathcal X}_{\eta,M}(\Lambda)}
\|a_tkw_V\|^{-1}\,dk
=
0.
\end{equation}
\end{prop}

\begin{proof}
Let
$$
\mathcal A_{<0}
=
\left\{
V\in\widehat{\mathcal X}_{\eta,M}(\Lambda):
\|w_V\|<1
\right\},
$$
and, for $\ell\ge0$, set
$$
\mathcal A_\ell
=
\left\{
V\in\widehat{\mathcal X}_{\eta,M}(\Lambda):
2^\ell\le\|w_V\|<2^{\ell+1}
\right\}.
$$
The set $\mathcal A_{<0}$ is finite, since the primitive Pl\"ucker
vectors of $\Lambda$-rational two-planes form a discrete subset of
$\wedge^2\Lambda$.

Let $V\in\mathcal A_\ell$. Since
$w_V\in\mathcal H_{\eta,M}$ and
$\eta\le2^{-M}\eta_0$, we have
$$
\begin{aligned}
\min\{\|w_V-\pi_cw_V\|,\|w_V-\pi_rw_V\|\}
&\le
\eta\|w_V\|^{-M} \le
\eta 2^{-\ell M} \le
\eta_0 2^{-(\ell+1)M}.
\end{aligned}
$$
Moreover, $\|w_V\|<2^{\ell+1}$. Hence
$$
\mathcal A_\ell
\subset
\mathcal Q_\Lambda
\left(
2^{\ell+1},
\eta_0 2^{-(\ell+1)M}
\right),
$$
and \eqref{fs} gives
$$
\#\mathcal A_\ell\le4
\qquad
(\ell\ge0).
$$

We first estimate the tail. By
\Cref{lem:critical-K-average-d2}, for every $L\ge0$ and every $t\ge0$,
$$
\begin{aligned}
\int_K
\sum_{\ell\ge L}
\sum_{V\in\mathcal A_\ell}
\|a_tkw_V\|^{-1}\,dk
&\le
C
\sum_{\ell\ge L}
\sum_{V\in\mathcal A_\ell}
\|w_V\|^{-1} \le
4C\sum_{\ell\ge L}2^{-\ell}
\ll
2^{-L}.
\end{aligned}
$$
This estimate is uniform in $t$.

For fixed $L$, the collection
$
\mathcal A_{<0}
\cup
\bigcup_{0\le\ell<L}\mathcal A_\ell
$
is finite, and every plane in it is non-isotropic. Therefore
\Cref{lem:critical-K-average-d2} gives
$$
\lim_{t\to\infty}
\int_K\|a_tkw_V\|^{-1}\,dk
=
0
$$
for each such $V$. Letting first $t\to\infty$ and then $L\to\infty$
proves \eqref{eq:nonisotropic-quasinull-vanishing-d2}.
\end{proof}

For $f\in C_c(\M_2(\mathbb R)\smallsetminus\{0\})$, define the modified Siegel transform
$$
\widehat f_{\operatorname{ni}}(h;\Lambda)
=
\sum_{\substack{0\ne v\in\Lambda\\
v\notin V\text{ for every }V\in\mathcal I_\Lambda}}
f(hv).
$$
Thus $\widehat f_{\operatorname{ni}}$ is obtained from the ordinary Siegel
transform by removing the lattice vectors contained in the exact rational
isotropic planes.

\begin{lem}[Modified Lipschitz principle]
\label{lem:siegel-height-d2}
Let $f\in C_c(\M_2(\mathbb R)\smallsetminus\{0\})$. There is a constant $C_f>0$ such that,
for every $h\in H$,
\begin{equation}\label{eq:exact-transform-domination-d2}
|\widehat f_{\operatorname{ni}}(h;\Lambda)|
\le
C_f\widehat\alpha_{\eta,M}(h;\Lambda)
+
C_f
\sum_{V\in\widehat{\mathcal X}_{\eta,M}(\Lambda)}
\|hw_V\|^{-1}.
\end{equation}
\end{lem}

\begin{proof}
This is the modified Lipschitz principle; compare
\cite[Lemma~2.7]{Kim}. Although the cited lemma is stated for smooth
functions, the continuous version follows by dominating $|f|$ by a function
in $C_c^\infty(\M_2(\mathbb R)\smallsetminus\{0\})$. The usual height
argument controls the contributions
from rational subspaces in degrees $1$ and $3$ and from non-quasi-null
rational two-planes. The remaining degree-$2$ contributions are precisely
those indexed by
$\widehat{\mathcal X}_{\eta,M}(\Lambda)$, which give the second term in
\eqref{eq:exact-transform-domination-d2}.
\end{proof}

We also need a subcritical moment of the ordinary height. Minkowski's
second theorem gives
$$
\alpha_2(\Delta)^{1/2}
\ll
\max\{1,\alpha_1(\Delta)\}.
$$
It follows that
$$
\alpha(h\Lambda)^{1/2}
\ll
\widehat\alpha_{\eta,M}(h;\Lambda).
$$
Consequently, if $\theta'>0$ is as in \Cref{thm:moment-d2}, then
\begin{equation}\label{eq:ordinary-subcritical-K-d2}
\sup_{t\ge0}
\int_K
\alpha(a_tk\Lambda)^{(1+\theta')/2}\,dk
<
\infty.
\end{equation}

We shall use the following equidistribution theorem of Shah.

\begin{theorem}\cite{shah}
\label{shah}
Suppose that $H\Lambda$ is dense in $X$. Then, for every $F\in C_c(X)$,
$$
\lim_{t\to\infty}
\int_KF(a_tk\Lambda)\,dk
=
\int_XF\,dm_X.
$$
\end{theorem}
We now pass to the modified Siegel
transform.

\begin{prop}[Modified Siegel-transform limit]
\label{prop:modified-siegel-limit-d2}
Let $\Lambda<\M_2(\mathbb R)$ be Diophantine and not
determinant-rational. Then, for every
$f\in C_c(\M_2(\mathbb R)\smallsetminus\{0\})$,
\[
\lim_{t\to\infty}
\int_K
\widehat f_{\operatorname{ni}}(a_tk;\Lambda)\,dk
=
\frac{1}{\covol(\Lambda)}
\int_{\M_2(\mathbb R)}f(x)\,dx.
\]
\end{prop}
\begin{proof} We may assume that $\Lambda\in X$ and $f\ge0$. Following
\cite[Section~7]{Kim}, we prove matching lower and upper bounds. Let
$\widehat f$ denote the ordinary Siegel transform, and put
\[
X_{>R}=\{\Delta\in X:\alpha(\Delta)>R\}.
\]
Choose a continuous function $h_R:X\to[0,1]$ such that
\[
\mathbf 1_{X_{>R+1}}\le h_R\le\mathbf 1_{X_{>R}}.
\]

For the lower bound, fix $\varepsilon>0$ and choose
$f_-\in C_c(\M_2(\mathbb R)\smallsetminus\{0\})$ such that
\[
0\le f_-\le f,\qquad
f_-=0\quad\text{on a neighborhood of }\{\det=0\},
\]
and
\[
\int_{\M_2(\mathbb R)}f_-(x)\,dx
\ge
\int_{\M_2(\mathbb R)}f(x)\,dx-\varepsilon.
\]
Such a function exists because $\{\det=0\}$ has Lebesgue measure zero.
Since every plane in $\mathcal I_\Lambda$ is contained in $\{\det=0\}$,
\[
\widehat{f_-}_{\operatorname{ni}}(h;\Lambda)
=
\widehat{f_-}(h\Lambda).
\]

By Siegel's mean-value formula \cite{siegel1945mean},
$\widehat{f_-}\in L^1(X)$. Taking $R$ sufficiently large, we may assume
that
\[
\int_X\widehat{f_-}(\Delta)h_R(\Delta)\,dm_X(\Delta)<\varepsilon.
\]
The function $\widehat{f_-}(1-h_R)$ is bounded, continuous, and compactly
supported on $X$. Since $H\Lambda$ is dense by
\Cref{thm:orbit-dichotomy-d2}, Shah's theorem and Siegel's mean-value
formula give
\[
\begin{aligned}
&\liminf_{t\to\infty}
\int_K\widehat f_{\operatorname{ni}}(a_tk;\Lambda)\,dk
\ge
\int_X\widehat{f_-}(\Delta)(1-h_R(\Delta))\,dm_X(\Delta)\\
&\ge
\int_{\M_2(\mathbb R)}f_-(x)\,dx-\varepsilon\ge 
\int_{\M_2(\mathbb R)}f(x)\,dx-2\varepsilon.
\end{aligned}
\]

For the upper bound, $\widehat f_{\operatorname{ni}}(h;\Lambda)\le
\widehat f(h\Lambda)$ pointwise. Applying Shah's theorem to
$\widehat f(1-h_R)$ gives
\[
\begin{aligned}
\limsup_{t\to\infty}
\int_K
\widehat f_{\operatorname{ni}}(a_tk;\Lambda)
(1-h_R(a_tk\Lambda))\,dk
&\le
\int_X\widehat f(\Delta)(1-h_R(\Delta))\,dm_X(\Delta)\\
&\le
\int_{\M_2(\mathbb R)}f(x)\,dx.
\end{aligned}
\]

It remains to estimate the cusp contribution. If
$h_R(a_tk\Lambda)\ne0$, then $\alpha(a_tk\Lambda)>R$, and hence
\[
\begin{aligned}
&\left(
\widehat\alpha_{\eta,M}(a_tk;\Lambda)
+
\alpha(a_tk\Lambda)^{1/2}
\right)h_R(a_tk\Lambda) \le
R^{-\theta'/2}
\left(
\widehat\alpha_{\eta,M}(a_tk;\Lambda)
+
\alpha(a_tk\Lambda)^{1/2}
\right)^{1+\theta'}.
\end{aligned}
\]
By \Cref{thm:moment-d2},
\eqref{eq:ordinary-subcritical-K-d2}, and
\eqref{eq:exact-transform-domination-d2},
\[
\begin{aligned}
&\limsup_{t\to\infty}
\int_K
\widehat f_{\operatorname{ni}}(a_tk;\Lambda)
h_R(a_tk\Lambda)\,dk\\
&\qquad\ll_f
R^{-\theta'/2}
+
\limsup_{t\to\infty}
\int_K
\sum_{V\in\widehat{\mathcal X}_{\eta,M}(\Lambda)}
\|a_tkw_V\|^{-1}\,dk
\ll_f R^{-\theta'/2},
\end{aligned}
\]
where \Cref{prop:nonisotropic-quasinull-vanishing-d2} shows that the
remaining limit is zero. Therefore
\[
\limsup_{t\to\infty}
\int_K\widehat f_{\operatorname{ni}}(a_tk;\Lambda)\,dk
\le
\int_{\M_2(\mathbb R)}f(x)\,dx+O_f(R^{-\theta'/2}).
\]
Letting $R\to\infty$ and $\varepsilon\to0$ proves the claim.
\end{proof}
\section{Counting and completion}
\label{sec:counting-completion}
We derive the determinant-counting asymptotic from the modified
Siegel-transform limit. A fiber kernel converts compactly supported
functions on $\M_2(\mathbb R)$ into functions of the largest singular
value and the determinant, yielding a counting limit on compact parameter
ranges. We then remove the range where $\sigma_1/T$ is small and count the
finitely many rational isotropic planes separately.

Fix a lattice $\Lambda<\M_2(\mathbb R)$ satisfying the hypotheses of
\Cref{thm:main}, and put
$$
\mathcal R(a,b;T)
=
\{X\in\M_2(\mathbb R):\|X\|<T,\ a<\det X<b\},
$$
and
$$
\Lambda_{\operatorname{ni}}
=
\{0\ne v\in\Lambda:
v\notin V\text{ for every }V\in\mathcal I_\Lambda\}.
$$

\subsection{The fiber kernel}

Let
$$
\M_2^+(\mathbb R)
=
\{X=(x_{ij})\in\M_2(\mathbb R):x_{11}>0\}.
$$
Extending by zero, we regard $C_c(\M_2^+(\mathbb R))$ as a subspace of
$C_c(\M_2(\mathbb R)\smallsetminus\{0\})$. For $r>0$, define
$$
\Theta(r,x,y,\zeta)
=
\begin{pmatrix}
r&y\\
x&(\zeta+xy)/r
\end{pmatrix}.
$$
For $f\in C_c(\M_2^+(\mathbb R))$, define
$J_f\in C_c((0,\infty)\times\mathbb R)$ by
\begin{equation}\label{eq:def-Jf-d2}
J_f(r,\zeta)
=
r^{-2}
\int_{\mathbb R^2}
f(\Theta(r,x,y,\zeta))\,dx\,dy.
\end{equation}

\begin{lem}
\label{lem:Jf-basic-d2}
For every $f\in C_c(\M_2^+(\mathbb R))$,
\begin{equation}\label{eq:Jf-volume-d2}
\int_{\M_2^+(\mathbb R)}f(x)\,dx
=
\int_{\mathbb R}\int_0^\infty
J_f(r,\zeta)\,r\,dr\,d\zeta.
\end{equation}
Conversely, for every $h\in C_c((0,\infty)\times\mathbb R)$, there is
$f\in C_c(\M_2^+(\mathbb R))$ such that $J_f=h$. If $h\ge0$, then $f$
may be chosen nonnegative.
\end{lem}

\begin{proof}
Write
$$
v=
\begin{pmatrix}
r&y\\
x&z
\end{pmatrix},
\qquad
\zeta=rz-xy.
$$
Since $dz=r^{-1}d\zeta$, formula \eqref{eq:Jf-volume-d2} follows from
Fubini's theorem.

Conversely, choose a nonnegative function
$\psi\in C_c(\mathbb R^2)$ with integral one, and define
$$
f(\Theta(r,x,y,\zeta))
=
r^2h(r,\zeta)\psi(x,y).
$$
The coordinates $(r,x,y,\zeta)$ are unique on
$\M_2^+(\mathbb R)$, and the support of $h$ is bounded away from $r=0$.
It follows that $f$ is compactly supported and $J_f=h$.
\end{proof}

For $T=e^{2t}$, $r>0$, and $|\zeta|\le T^2r^2$, put
$$
D_T(r,\zeta)
=
\operatorname{diag}\left(Tr,\frac{\zeta}{Tr}\right),
\qquad
J_{f,T}(r,\zeta)
=
T^2\int_Kf(a_tkD_T(r,\zeta))\,dk.
$$
Set $J_{f,T}(r,\zeta)=0$ when $|\zeta|>T^2r^2$.

\begin{lem}[Orbital kernel]
\label{lem:orbital-kernel-d2}
For every $f\in C_c(\M_2^+(\mathbb R))$,
\[
J_{f,T}\to \frac{1}{2\pi^2}J_f
\]
uniformly on compact subsets of $(0,\infty)\times\mathbb R$ as
$T\to\infty$. Moreover, for all sufficiently large $T$, the supports of
$J_{f,T}$ lie in a fixed compact subset of
$(0,\infty)\times\mathbb R$.
\end{lem}
\begin{proof}
Write elements of $\SO(2)$ as
$
k_\vartheta
=
\begin{pmatrix}
\cos\vartheta&-\sin\vartheta\\
\sin\vartheta&\cos\vartheta
\end{pmatrix}.
$
A direct calculation gives
$$
\begin{aligned}
\bigl(a_t(k_\vartheta,k_\varphi)D_T(r,\zeta)\bigr)_{11}
&=
r\cos\vartheta\cos\varphi
+
\frac{\zeta}{T^2r}\sin\vartheta\sin\varphi,\\
\bigl(a_t(k_\vartheta,k_\varphi)D_T(r,\zeta)\bigr)_{12}
&=
Tr\cos\vartheta\sin\varphi
-
\frac{\zeta}{Tr}\sin\vartheta\cos\varphi,\\
\bigl(a_t(k_\vartheta,k_\varphi)D_T(r,\zeta)\bigr)_{21}
&=
Tr\sin\vartheta\cos\varphi
-
\frac{\zeta}{Tr}\cos\vartheta\sin\varphi,\\
\bigl(a_t(k_\vartheta,k_\varphi)D_T(r,\zeta)\bigr)_{22}
&=
T^2r\sin\vartheta\sin\varphi
+
\frac{\zeta}{r}\cos\vartheta\cos\varphi.
\end{aligned}
$$

Fix a compact subset of $(0,\infty)\times\mathbb R$. If the displayed
matrix belongs to $\supp f$, boundedness of its second and third entries
gives
$$
|\cos\vartheta\sin\varphi|
+
|\sin\vartheta\cos\varphi|
\ll T^{-1},
$$
while boundedness of the fourth entry gives
$|\sin\vartheta\sin\varphi|\ll T^{-2}$. Thus each angle lies within
$O(T^{-1})$ of $0$ or $\pi$. Only neighborhoods of
$(0,0)$, $(\pi,\pi)$, $(0,\pi)$, and $(\pi,0)$ can contribute. The last
two are excluded for large $T$, since the upper-left entry is then
negative, whereas $\supp f$ is compactly contained in $\{x_{11}>0\}$.

Near $(0,0)$, put
$$
x=Tr\vartheta,
\qquad
y=Tr\varphi.
$$
The support condition bounds $x$ and $y$. Uniformly for bounded $x,y$ and
for $(r,\zeta)$ in the fixed compact set,
$$
a_t(k_\vartheta,k_\varphi)D_T(r,\zeta)
=
\Theta(r,x,y,\zeta)+O(T^{-1}).
$$
The same expansion holds near $(\pi,\pi)$ after changing the signs of
$x$ and $y$. Moreover,
$$
T^2\frac{d\vartheta\,d\varphi}{(2\pi)^2}
=
\frac{1}{(2\pi)^2}r^{-2}\,dx\,dy.
$$
The two contributing neighborhoods therefore give
$
J_{f,T}(r,\zeta)
\to
\frac{1}{2\pi^2}J_f(r,\zeta)
$
uniformly on compact subsets.

It remains to prove the support assertion. The claim is immediate if
$f=0$, so assume otherwise. Since $\supp f$ is a compact subset of
$\M_2^+(\mathbb R)$, we have
$$
x_{11}\gg_f1,
\qquad
\|X\|\ll_f1
\qquad(X\in\supp f).
$$
If $J_{f,T}(r,\zeta)\ne0$, then $|\zeta|\le T^2r^2$, and the upper-left
entry in the displayed formula has absolute value at most $2r$. Hence
$r\gg_f1$.

By equivalence of norms and the smallest-singular-value bound for the
action of $a_tk$,
$$
\|a_tkD_T(r,\zeta)\|
\gg_{\|\cdot\|}
T^{-1}\|D_T(r,\zeta)\|_{\operatorname{Euc}}
\ge r,
$$
where $\|\cdot\|_{\operatorname{Euc}}$ is induced by the trace inner
product. Since $a_tkD_T(r,\zeta)\in\supp f$ for some $k\in K$, it follows
that $r\ll_f1$. Finally, $\zeta$ is the determinant of a matrix in
$\supp f$, so $|\zeta|\ll_f1$. This proves the support assertion.
\end{proof}

\subsection{The parameter-space limit}

Let $\sigma_1(v)\ge\sigma_2(v)\ge0$ denote the singular values of $v\in \M_2(\mathbb R)$.

\begin{prop}
\label{prop:parameter-limit-d2}
For every $h\in C_c((0,\infty)\times\mathbb R)$, we have
\begin{equation}\label{eq:parameter-limit-d2}
\lim_{T\to\infty}
T^{-2}
\sum_{v\in\Lambda_{\operatorname{ni}}}
h\left(T^{-1}\sigma_1(v),\det v\right)
=
\frac{2\pi^2}{\covol(\Lambda)}
\int_{\mathbb R}\int_0^\infty
h(r,\zeta)\,r\,dr\,d\zeta.
\end{equation}
The corresponding Lebesgue-measure limit is
$$
\lim_{T\to\infty}
T^{-2}
\int_{\M_2(\mathbb R)}
h\left(T^{-1}\sigma_1(x),\det x\right)\,dx
=
2\pi^2
\int_{\mathbb R}\int_0^\infty
h(r,\zeta)\,r\,dr\,d\zeta.
$$
\end{prop}

\begin{proof}
 By \Cref{lem:Jf-basic-d2}, choose
$f\in C_c(\M_2^+(\mathbb R))$ such that $J_f=h$. Singular-value decomposition and the
bi-invariance of $dk$ give
$$
T^{-2}
J_{f,T}\left(T^{-1}\sigma_1(v),\det v\right)
=
\int_Kf(a_tkv)\,dk.
$$
After summing over $v\in\Lambda_{\operatorname{ni}}$,
\Cref{prop:modified-siegel-limit-d2} gives
$$
\lim_{T\to\infty}
T^{-2}
\sum_{v\in\Lambda_{\operatorname{ni}}}
J_{f,T}\left(T^{-1}\sigma_1(v),\det v\right)
=
\frac{1}{\covol(\Lambda)}
\int_{\M_2(\mathbb R)}f(x)\,dx.
$$

Let $\mathcal C\subset(0,\infty)\times\mathbb R$ be a compact set
containing $\supp J_f$ and the supports of $J_{f,T}$ for all sufficiently
large $T$, as supplied by \Cref{lem:orbital-kernel-d2}. Choose
$0\le h_0\in C_c((0,\infty)\times\mathbb R)$ such that $h_0\ge1$ on
$\mathcal C$, and choose $f_0\ge0$ with $J_{f_0}=h_0$. Uniform convergence
gives
$$
J_{f_0,T}\ge\frac{1}{4\pi^2}
\qquad\text{on }\mathcal C
$$
for all sufficiently large $T$. Applying the preceding limit to $f_0$
therefore gives
$$
T^{-2}
\#\left\{
v\in\Lambda_{\operatorname{ni}}:
\left(T^{-1}\sigma_1(v),\det v\right)\in\mathcal C
\right\}
\ll1.
$$
Consequently,
$$
\begin{aligned}
&T^{-2}
\left|
\sum_{v\in\Lambda_{\operatorname{ni}}}
\left(J_{f,T}-\frac{1}{2\pi^2}J_f\right)
\left(T^{-1}\sigma_1(v),\det v\right)
\right|\\
&\qquad\le
\sup_{\mathcal C}|J_{f,T}-\frac{1}{2\pi^2}J_f|\,
T^{-2}
\#\left\{
v\in\Lambda_{\operatorname{ni}}:
\left(T^{-1}\sigma_1(v),\det v\right)\in\mathcal C
\right\}
=o(1).
\end{aligned}
$$
The preceding estimate and \eqref{eq:Jf-volume-d2} prove
\eqref{eq:parameter-limit-d2}.

For Lebesgue measure, integrate the exact orbital identity over
$\M_2(\mathbb R)$:
$$
\begin{aligned}
&\tfrac{1}{T^2}
\int_{\M_2(\mathbb R)}
J_{f,T}\left(T^{-1}\sigma_1(x),\det x\right)\,dx=
\int_K\int_{\M_2(\mathbb R)}
f(a_tkx)\,dx\,dk
=
\int_{\M_2(\mathbb R)}f(x)\,dx.
\end{aligned}
$$
The same uniform-convergence argument proves the Lebesgue-measure limit.
\end{proof}

\subsection{The volume asymptotic}

\begin{prop}[Volume asymptotic]
\label{prop:volume-d2}
For every $a<b$,
$$
\vol\left(\mathcal R(a,b;T)\right)
\sim
C_{\|\cdot\|}(b-a)T^2,
$$
where
$
C_{\|\cdot\|}
=
2\pi^2
\int_{\{r>0:\|\operatorname{diag}(r,0)\|<1\}}
r\,dr.
$
\end{prop}

\begin{proof}
The Weyl integration formula in signed singular-value coordinates, with
probability Haar measure on each copy of $\SO(2)$, is
$$
dx
=
2\pi^2(r^2-s^2)\,dr\,ds\,dk_1\,dk_2,
\qquad
r>0,\quad -r<s<r.
$$
In these coordinates, $\det x=rs$. Put $u=rs$, so that
$s=u/r$. Since the norm is $K$-invariant,
$$
\vol\mathcal R(a,b;T)
=
2\pi^2
\int_a^b
\int_{\substack{r>|u|^{1/2}\\
\|\operatorname{diag}(r,r^{-1}u)\|<T}}
\left(r-r^{-3}{u^2}\right)\,dr\,du.
$$

Set $r=T\rho$. After dividing by $T^2$, the integrand becomes
$
\rho-\frac{u^2}{T^4\rho^3},
$
which is nonnegative and bounded above by $\rho$. The
range of $\rho$ is contained in a fixed bounded interval, independently of
$T$ and $u\in[a,b]$. Moreover, for every $\rho>0$ and almost every $u$,
$$
\left\|
\operatorname{diag}
\left(\rho,({T^2\rho})^{-1} u\right)
\right\|
\to
\|\operatorname{diag}(\rho,0)\|.
$$
Dominated convergence therefore proves the claim.
\end{proof}

\subsection{Compact windows and removal of tails}

For a compact interval $J\subset(0,\infty)$, set
$$
\mathcal R_J(a,b;T)
=
\left\{
v\in\mathcal R(a,b;T):
T^{-1}{\sigma_1(v)}\in J
\right\}.
$$

\begin{lem}[Compact-window count]
\label{lem:compact-window-counting-d2}
For every compact interval $J\subset(0,\infty)$,
$$
\#\bigl(\Lambda_{\operatorname{ni}}\cap\mathcal R_J(a,b;T)\bigr)
=
\frac{1}{\covol(\Lambda)}
\vol\mathcal R_J(a,b;T)
+
o(T^2).
$$
\end{lem}

\begin{proof}
Put
$
g(r)=\|\operatorname{diag}(r,0)\|.
$
For $\delta>0$, let
$$
J_\delta^-
=
\{r\in J:g(r)\le1-\delta\},
\qquad
J_\delta^+
=
\{r\in J:g(r)<1+\delta\}.
$$
Uniformly for $u\in[a,b]$ and $r\in J$,
$$
\left\|
\operatorname{diag}
\left(r, ({T^2r})^{-1}{u}\right)
\right\|
\to
g(r).
$$
Hence, for all sufficiently large $T$,
$$
\begin{aligned}
\left\{
v:
T^{-1}{\sigma_1(v)}\in J_\delta^-,
\ a<\det v<b
\right\}
&\subset
\mathcal R_J(a,b;T)\\
&\subset
\left\{
v: T^{-1} {\sigma_1(v)}\in J_\delta^+,
\ a<\det v<b
\right\}.
\end{aligned}
$$
By the usual approximation from above and below,
\Cref{prop:parameter-limit-d2} and its Lebesgue-measure version apply to
the indicators of $J_\delta^-\times(a,b)$ and
$J_\delta^+\times(a,b)$. Since
$
g(r)=r\|\operatorname{diag}(1,0)\|,
$
we have
$
\int_{J_\delta^+\smallsetminus J_\delta^-}r\,dr
\to0
$
as $\delta\to0$. Letting $\delta\to0$ proves the result.
\end{proof}

\begin{lem}[Uniform dyadic-shell bound]
\label{lem:dyadic-shell-bound-d2}
For every bounded interval $I\subset\mathbb R$, there is a constant
$C_I>0$ such that, for every $S\ge1$,
$$
\#\left\{
v\in\Lambda_{\operatorname{ni}}:
{S}/{2}<\sigma_1(v)\le S,\ \det v\in I
\right\}
\le
C_IS^2.
$$
\end{lem}

\begin{proof}
Choose a nonnegative function
$h\in C_c((0,\infty)\times\mathbb R)$ supported in
$[1/3,2]\times I'$, where $I'$ is a compact interval containing $I$, and
such that $h=1$ on $[1/2,1]\times I$. By
\Cref{prop:parameter-limit-d2},
$$
S^{-2}
\sum_{v\in\Lambda_{\operatorname{ni}}}
h\left(S^{-1}{\sigma_1(v)},\det v\right)
\ll_h1
$$
for all sufficiently large $S$. The remaining bounded range of $S$ is
absorbed into the constant, since the corresponding set of lattice points
is finite.
\end{proof}

\begin{lem}[Removing the small-$\sigma_1$ tail]
\label{lem:small-sigma-tail-d2}
For every $a<b$,
$$
\lim_{\delta\to0}
\limsup_{T\to\infty}
T^{-2}
\vol\left\{
v\in\mathcal R(a,b;T):
\sigma_1(v)<\delta T
\right\}
=
0;
$$

$$
\lim_{\delta\to0}
\limsup_{T\to\infty}
T^{-2}
\#\left\{
v\in\Lambda_{\operatorname{ni}}\cap\mathcal R(a,b;T):
\sigma_1(v)<\delta T
\right\}
=
0.
$$
\end{lem}

\begin{proof}
The volume estimate follows from the proof of
\Cref{prop:volume-d2}. For the lattice estimate, decompose $(0,\delta T]$ into the dyadic shells
$$
2^{-j-1}\delta T
<
\sigma_1
\le
2^{-j}\delta T.
$$
As long as $2^{-j}\delta T\ge1$,
\Cref{lem:dyadic-shell-bound-d2} bounds the number of points in this shell
by
$
C(2^{-j}\delta T)^2.
$
The remaining points satisfy $\sigma_1(v)<1$ and form a fixed finite set.
Summing the geometric series gives
$
O(\delta^2T^2)+O(1),
$
uniformly in $T$.
\end{proof}

\subsection{The nonsingular count}
\label{sec:nonsingular-d2}

We first count the lattice points outside the exact rational isotropic
planes, allowing determinant zero. We then show that the remaining
zero-determinant points are negligible.

\begin{prop}
\label{prop:nonisotropic-window-d2}
For every $a<b$,
$$
\#\left\{
v\in\Lambda_{\operatorname{ni}}:
\|v\|<T,\ a<\det v<b
\right\}
\sim
\frac{C_{\|\cdot\|}}{\covol(\Lambda)}
(b-a)T^2.
$$
\end{prop}

\begin{proof}
Choose $R_{\|\cdot\|}>0$ such that
$\sigma_1(x)\le R_{\|\cdot\|}\|x\|$ for
$x\in\M_2(\mathbb R)$.
For fixed $\delta>0$, apply
\Cref{lem:compact-window-counting-d2} with
$J=[\delta,R_{\|\cdot\|}]$. The complement of this compact window is
negligible for both lattice points and volume by
\Cref{lem:small-sigma-tail-d2}. Letting first $T\to\infty$ and then
$\delta\to0$, and using \Cref{prop:volume-d2}, proves the assertion.
\end{proof}

\begin{lem}
\label{lem:nonisotropic-zero-negligible-d2}
One has
$$
\#\left\{
v\in\Lambda_{\operatorname{ni}}:
\|v\|<T,\ \det v=0
\right\}
=
o(T^2).
$$
\end{lem}

\begin{proof}
For every $\varepsilon>0$,
\Cref{prop:nonisotropic-window-d2}, applied to
$(-\varepsilon,\varepsilon)$, gives
$$
\begin{aligned}
\limsup_{T\to\infty}
T^{-2}
\#\left\{
v\in\Lambda_{\operatorname{ni}}:
\|v\|<T,\ |\det v|<\varepsilon
\right\}
\le
\frac{2C_{\|\cdot\|}}{\covol(\Lambda)}
\varepsilon.
\end{aligned}
$$
Letting $\varepsilon\to0$ proves the claim.
\end{proof}

\begin{prop}[Nonsingular asymptotic]
\label{prop:nonsingular-d2}
For every $a<b$,
$$
N_\Lambda^\times(a,b;T)
\sim
\frac{C_{\|\cdot\|}}{\covol(\Lambda)}
(b-a)T^2.
$$
\end{prop}

\begin{proof}
Every nonzero-determinant vector lies outside the planes in
$\mathcal I_\Lambda$. Therefore
$$
\begin{aligned}
N_\Lambda^\times(a,b;T)
={}&
\#\left\{
v\in\Lambda_{\operatorname{ni}}:
\|v\|<T,\ a<\det v<b
\right\}\\
&-
\mathbf 1_{\{0\in(a,b)\}}
\#\left\{
v\in\Lambda_{\operatorname{ni}}:
\|v\|<T,\ \det v=0
\right\}.
\end{aligned}
$$
The assertion follows from
\Cref{prop:nonisotropic-window-d2,lem:nonisotropic-zero-negligible-d2}.
\end{proof}

\subsection{The singular contribution}
\label{sec:singular-d2}

\begin{prop}[Singular asymptotic]
\label{prop:singular-d2}
One has
$$
\lim_{T\to\infty}
T^{-2}
\#\{v\in\Lambda:\|v\|<T,\ \det v=0\}
=
c_\Lambda^{\operatorname{sing}},
$$
where $c_\Lambda^{\operatorname{sing}}$ is given by
\eqref{eq:sing-explicit}.
\end{prop}

\begin{proof}
For each $V\in\mathcal I_\Lambda$, the standard lattice-point asymptotic
for dilates of a bounded convex planar region gives
$$
\#\{v\in\Lambda\cap V:\|v\|<T\}
=
\frac{\vol_V\{v\in V:\|v\|<1\}}
{\covol_V(\Lambda\cap V)}
T^2
+
o(T^2).
$$
If $V,W\in\mathcal I_\Lambda$ are distinct, then
$\dim(V\cap W)\le1$, so the lattice points in their intersection contribute
$O(T)$. Summing over the finite set $\mathcal I_\Lambda$ and using
\Cref{lem:nonisotropic-zero-negligible-d2} proves the result.
\end{proof}

\begin{cor}[Positivity of the singular constant]
\label{cor:singular-positivity-d2}
One has $c_\Lambda^{\operatorname{sing}}>0$ if and only if $\Lambda$ contains a
rank-two $\mathbb Z$-submodule on which the determinant vanishes
identically.
\end{cor}

\begin{proof}
Such a submodule spans a $\Lambda$-rational two-plane contained in
$\{\det=0\}$, and hence an element of $\mathcal I_\Lambda$. Conversely,
every $V\in\mathcal I_\Lambda$ contains the rank-two lattice
$\Lambda\cap V$.
\end{proof}

\subsection{Completion of the proof}

For every $T>0$,
$$
N_\Lambda(a,b;T)
=
N_\Lambda^\times(a,b;T)
+
\mathbf 1_{\{0\in(a,b)\}}
\#\{v\in\Lambda:\|v\|<T,\ \det v=0\}.
$$
The asymptotic formula \eqref{eq:main-full} now follows from
\Cref{prop:nonsingular-d2,prop:singular-d2}, and the characterization of
the positivity of the singular term follows from
\Cref{cor:singular-positivity-d2}. This completes the proof of
\Cref{thm:main}.
\def\cprime{$'$}

\end{document}